\documentclass[11pt]{article}
\usepackage[utf8]{inputenc}
\usepackage{enumitem}
\usepackage{amsmath,amssymb,mathtools}
\usepackage{amsthm}
\usepackage{listings}
\usepackage{soul,color,xcolor}
\usepackage[a4paper]{geometry}
\usepackage{tikz}
\usepackage{float}
\usepackage{mathrsfs}
\usepackage{enumitem}
\usepackage{algorithm,algorithmic}
\allowdisplaybreaks[3]
\usepackage{bbm}
\usepackage{microtype}
\usepackage[a4paper]{geometry}
\usepackage{pifont}
\usepackage{latexsym}
\newtheorem{lemma}{\textbf{Lemma}}
\newtheorem{thm}{\textbf{Theorem}}

\newtheorem{ass}{\textbf{Assumption}}

\newtheorem{rem}{Remark}

\newcommand{\beq}{\begin{equation}}
\newcommand{\eeq}{\end{equation}}
\newcommand{\beqa}{\begin{eqnarray}}
\newcommand{\eeqa}{\end{eqnarray}}
\newcommand{\beqas}{\begin{eqnarray*}}
\newcommand{\eeqas}{\end{eqnarray*}}
\newcommand{\ba}{\begin{array}}
\newcommand{\ea}{\end{array}}
\newcommand{\bi}{\begin{itemize}}
\newcommand{\ei}{\end{itemize}}

\newcommand\dom[1]{\mathrm{dom}\,#1}
\newcommand\res{\mathrm{res}}
\newcommand\dist{\mathrm{dist}}

\newcommand\rr{\mathbb{R}}

\newcommand\oo[1]{\mathcal{O}\left(#1\right)}
\newcommand\cS{\mathcal{S}}
\newcommand\hS{\mathcal{D}}

\newcommand\nn{N}
\newcommand\mm{\bar M}

\newcommand\tf{\tilde f}
\newcommand\tg{\tilde\gamma}

\newcommand\td{\tilde\Delta}

\newcommand\cO{\mathcal{O}}

\newcommand\lref[1]{Lemma \ref{#1}}

\newcommand\tref[1]{Theorem \ref{#1}}
\newcommand\cref[1]{Corollary \ref{#1}}
\newcommand\aref[1]{Algorithm \ref{#1}}
\newcommand\mref[1]{(\ref{#1})}

\newcommand{\mcK}{{\mathcal K}}

\newcommand{\argmin}{\arg\min}
\newcommand{\prox}{\mathrm{prox}}
\def\bbK{{\mathbb K}}
\def\bbT{{\mathbb T}}

\def\br{{\bar r}}
\def\cl{{\mathrm{cl}}}
\def\cT{{T}}
\def\cX{{\cal X}}
\def\cY{{\cal Y}}

\def\hS{{\widehat \cS}}
\def\cN{{\mathcal N}}
\def\cQ{{\mathcal Q}}
\def\hQ{{\widehat\cQ}}
\def\mT{{\mathscr{T}}}
\def\tP{{\tilde P}}
\usepackage{soul}

\title{Primal-dual extrapolation methods for monotone inclusions under local Lipschitz continuity}

\author{
Zhaosong Lu
\thanks{
Department of Industrial and Systems Engineering, University of Minnesota, USA (email: {\tt zhaosong@umn.edu}, {\tt mei00035@umn.edu}). This work was partially supported by NSF Award IIS-2211491.}
\and
Sanyou Mei
\footnotemark[1]
}

\date{June 1, 2022 (Revised: August 30, 2024)}

\begin{document}
\maketitle

\begin{abstract}
In this paper we consider a class of monotone inclusion (MI) problems of finding a zero of the sum of two monotone operators,  in which one operator is maximal monotone while the other is \emph{locally Lipschitz} continuous. We propose primal-dual extrapolation methods to solve them using a point and operator extrapolation technique, whose parameters are chosen by a backtracking line search scheme. The proposed methods enjoy an operation complexity of $\oo{\log \varepsilon^{-1}}$ and $\oo{\varepsilon^{-1}\log \varepsilon^{-1}}$, measured by the number of fundamental operations consisting only of evaluations of one operator and resolvent of the other operator, for finding an $\varepsilon$-residual solution of strongly and non-strongly MI problems, respectively.  The latter complexity significantly improves the previously best operation complexity $\oo{\varepsilon^{-2}}$. As a byproduct, complexity results of the primal-dual  extrapolation methods are also obtained for finding an $\varepsilon$-KKT or $\varepsilon$-residual solution of convex conic optimization, conic constrained saddle point, and variational inequality problems under {\it local Lipschitz} continuity.  We provide preliminary numerical results to demonstrate the performance of the proposed methods. 
\end{abstract}


\medskip


\section{Introduction}
A broad range of optimization, saddle point (SP), and variational inequality (VI)  problems can be solved as a monotone inclusion (MI) problem, namely, finding a point $x$ such that $0\in\mT(x)$, where $\mT: \rr^n\rightrightarrows\rr^n$ is a maximal monotone set-valued (i.e., point-to-set) operator (see Section \ref{notation} for the definition of monotone and maximal monotone operators). In this paper we consider a class of MI problems as follows:
\begin{equation}\label{MI}
\mbox{find }x\in\rr^n \mbox{ such that } 0\in(F+B)(x),
\end{equation}
where $B:\rr^n\rightrightarrows\rr^n$ is a maximal monotone set-valued operator with a 
 nonempty domain denoted by $\dom{B}$, and 
$F$ is a monotone point-valued (i.e., point-to-point) operator on $\cl(\dom{B})$.
It shall be mentioned that $\dom{B}$ is possibly \emph{unbounded}. We make the following {\it additional} assumptions throughout this paper.
\begin{ass}\label{ass0}
\begin{enumerate}[label=(\alph*)]
\item Problem \eqref{MI} has at least one solution.
\item $F+B$ is monotone on $\dom{B}$ with a monotonicity parameter $\mu \ge 0$ such that
\begin{equation}\label{monotone}
\langle u-v,x-y\rangle\geq\mu\|x-y\|^2 \qquad \forall x,y\in\dom{B}, u\in(F+B)(x), v\in(F+B)(y).
\end{equation}
\item $F$ is locally Lipschitz continuous on $\mathrm{cl}(\dom{B})$.\footnote{See Section \ref{notation} for the definition of local Lipschitz continuity of a mapping on a closed set.}
\item The resolvent of $\gamma B$ can be exactly evaluated for any $\gamma>0$.
\end{enumerate}
\end{ass}
The \emph{local Lipschitz continuity} of $F$ on $\mathrm{cl}(\dom{B})$ is generally weaker than the (global) Lipschitz continuity of $F$ on $\mathrm{cl}(\dom{B})$ usually imposed in the literature. Moreover, it can sometimes be easily verified. For example, if $F$ is continuously differentiable on $\mathrm{cl}(\dom{B})$, it is clearly locally Lipschitz continuous there. In addition, by the maximal monotonicity of $B$ and Assumptions \ref{ass0}(b) and \ref{ass0}(c), it can be observed that $F+B$ is maximal monotone (e.g., see \cite[Proposition A.1]{monteiro2011complexity}) and it is also strongly monotone when $\mu>0$.

Several special cases of problem \eqref{MI} have been considerably studied in the literature. For example, when $F$ is \emph{cocoercive}\footnote{$F$ is cocoercive if there exists some $\sigma>0$ such that $\langle F(x)-F(y), x-y\rangle \ge \sigma \|F(x)-F(y)\|^2$ for all $x,y\in \dom{F}$. It can be observed that if $F$ is cocoercive, then it is monotone and Lipschitz continuous on $\dom{F}$.} problem \eqref{MI} can be suitably solved by  a splitting inertial proximal method \cite{moudafi2003convergence}, a Halpern fixed-point splitting method \cite{tran2022connection}, and also the classical forward-backward splitting (FBS) method \cite{lions1979splitting,passty1979ergodic} that generates a solution sequence $\{x^k\}$ according to
\[
x^{k+1}=\left(I+\gamma_k B\right)^{-1}\big(x^k-\gamma_k F(x^k)\big) \quad \forall k \ge 1.
\]
In addition, a modified FBS (MFBS) method \cite{tseng2000modified}, its variant \cite{monteiro2010complexity},  an inertial forward-backward-forward splitting method \cite{boct2016inertial}, and an extra anchored gradient method \cite[Algorithm 3]{kovalev2022first} were proposed for \eqref{MI} with $F$ being \emph{Lipschitz continuous}. It shall be mentioned that operation complexity bounds of $\oo{\varepsilon^{-2}}$ and $\oo{\varepsilon^{-1}}$, measured by the number of fundamental operations consisting of evaluations of $F$ and resolvent of $B$, were respectively established for the variant of MFBS method \cite[Theorem 4.6]{monteiro2010complexity} and the extra anchored gradient method \cite[Theorem 2]{kovalev2022first} for finding an $\varepsilon$-residual solution\footnote{An $\varepsilon$-residual solution of problem \mref{MI} is a point $x\in\dom{B}$ satisfying $\res_{F+B}(x) \le \varepsilon$, where $\res_{F+B}(x) = \inf\{\|v\|: v\in (F+B)(x)\}$.} of \eqref{MI} with Lipschitz continuous $F$.

There has been little algorithmic development for solving problem \mref{MI} with locally Lipschitz continuous $F$. Indeed, the MFBS method \cite{tseng2000modified} and the forward-reflected-backward splitting (FRBS) method \cite[Algorithm 3.1]{malitsky2020forward} appear to be the only existing methods for solving this problem. The MFBS method modifies the classical FRBS method in the spirit of the extragradient method \cite{Kor76} for monotone variational inequalities, while the FRBS method modifies the forward term in the classical FBS method using an operator extrapolation technique that has been popularly used to design algorithms for solving optimization, SP, and VI problems (e.g., \cite{huang2021unifying, huang2022new, kotsalis2020simple, Mok20-2, popov1980modification}). Specifically, the FRBS method generates a solution sequence $\{x^k\}$ according to
\beq \label{FBS}
x^{k+1}=\left(I+\gamma_k B\right)^{-1}\left(x^k-\gamma_k F(x^k)-\gamma_{k-1}(F(x^k)-F(x^{k-1}))\right) \quad \forall k \ge 1
\eeq
for a suitable choice of stepsizes $\{\gamma_k\}$. Global convergence to a solution of problem \eqref{MI} are established for these methods in \cite{malitsky2020forward,tseng2000modified}, respectively. Moreover, it can be shown that the FRBS method enjoys an operation complexity of $\cO(\varepsilon^{-2})$ for finding an $\varepsilon$-residual solution of \eqref{MI} by using \cite[equation (2.14), Theorem 3.4, and Lemmas 3.2 and 3.3]{malitsky2020forward}, although this result is not established in \cite{malitsky2020forward}. In addition, when $B=\partial g$, where $g$ is a proper closed convex function,  an adaptive golden ratio algorithm was proposed in \cite[Algorithm 1]{malitsky2020golden}. While \cite{malitsky2020golden} did not specifically study the operation complexity of this algorithm for finding an $\varepsilon$-residual solution of \eqref{MI} with $B=\partial g$, it can be shown that the algorithm achieves an operation complexity of $\cO(\varepsilon^{-2})$ for such a solution by using \cite[equation (34) and Lemma 2]{malitsky2020golden}.

As seen from the above discussion, there is a significant gap between the best operation complexities of $\cO(\varepsilon^{-2})$ and  $\cO(\varepsilon^{-1})$ for finding an $\varepsilon$-residual solution of \eqref{MI} and its special case with Lipschitz continuous $F$, which are achieved by the FRBS method \cite{malitsky2020forward} and the extra anchored gradient method \cite{kovalev2022first}, respectively. 
To significantly shorten this gap, in this paper we propose new variants of FBS method, called {\it primal-dual} (PD) {\it extrapolation} methods, for finding an $\varepsilon$-residual solution of \eqref{MI} with complexity guarantees. In particular, we first propose a PD extrapolation method for solving a strongly MI problem, namely, problem \eqref{MI} with $\mu>0$, by modifying the forward term in the FBS method using a \emph{point and operator extrapolation technique} that has recently been used to design algorithms for solving stochastic VI problems in  \cite{huang2022new} and problem \eqref{MI} with Lipschitz continuous $F$ in \cite{malitsky2020forward}.  Specifically, this PD extrapolation method generates a solution sequence $\{x^k\}$ according to
\[
x^{k+1}=\left(I+\gamma_k B\right)^{-1}\left(x^k+\alpha_k(x^k-x^{k-1})- \gamma_k [F(x^k)+\beta_k (F(x^k)-F(x^{k-1}))]\right) \quad \forall k \ge 1,
\]
where the sequences $\{\alpha_k\}$, $\{\beta_k\}$ and $\{\gamma_k\}$ are updated by a backtracking line search scheme (see \aref{ALG}). We show that this PD extrapolation method enjoys an operation complexity of $\oo{\log \varepsilon^{-1}}$ for finding an $\varepsilon$-residual solution of \eqref{MI} with $\mu>0$. We then propose another PD extrapolation method for solving a non-strongly MI problem, namely,  problem \mref{MI} with $\mu=0$ by applying the above PD extrapolation method to approximately solve a sequence of strongly MI problems $0\in (F_k+B)(x)$ with $F_k$ being a perturbation of $F$ (see \aref{PPA}). We show that the resulting PD extrapolation method enjoys an operation complexity of $\oo{\varepsilon^{-1}\log \varepsilon^{-1}}$ for finding an $\varepsilon$-residual solution of  problem \eqref{MI} with $\mu=0$, which 
significantly improves the previously best operation complexity $\cO(\varepsilon^{-2})$ achieved by  the FRBS method  \cite{malitsky2020forward}. 

The main contributions of our paper are summarized as follows.

\bi
\item Primal-dual extrapolation methods are proposed for the MI problem \eqref{MI} with locally Lipschitz continuous $F$, which enjoy several attractive features: (i) they are applicable to a broad range of problems since only local rather than global Lipschitz continuity of $F$ is required; (ii) they adopt a point and operator extrapolation technique with fundamental operations consisting only of evaluations of $F$ and resolvent of $B$; (iii) they are equipped with a verifiable termination criterion and output an $\varepsilon$-residual solution of problem \eqref{MI} with complexity guarantees.   

\item We show that an $\varepsilon$-residual solution of problem \eqref{MI} with locally Lipschitz continuous $F$ can be found by our methods with an operation complexity of $\oo{\log \varepsilon^{-1}}$ and  $\oo{\varepsilon^{-1}\log \varepsilon^{-1}}$ for $\mu>0$ and $\mu=0$, respectively. The latter complexity significantly improves the previously best operation complexity $\cO(\varepsilon^{-2})$ achieved by the FRBS method \cite{malitsky2020forward}.  

\item The applications of our proposed methods to convex conic optimization, conic constrained SP, and VI problems are studied. Best complexity results for finding an $\varepsilon$-KKT or $\varepsilon$-residual solution of these problems under local Lipschitz continuity are obtained.
\ei

The rest of this paper is organized as follows. In Section \ref{notation} we introduce some notation and terminology. In Sections~\ref{PDE1} and \ref{PDE2}, we propose PD extrapolation methods for problem \eqref{MI} with $\mu>0$ and $\mu=0$, respectively, and study their complexity. In Section \ref{app}, we study the applications of the PD extrapolation methods for solving convex conic optimization, conic constrained saddle point, and variational inequality problems. In addition, we present some preliminary numerical results and the proofs of the main results in Sections \ref{sec:exp} and \ref{sec:proof}, respectively.  Finally, we make some concluding remarks in Section~\ref{sec:cr}.

\subsection{Notation and terminology} \label{notation}
The following notations will be used throughout this paper. Let $\rr^n$ denote the Euclidean space of dimension $n$, $\langle\cdot,\cdot\rangle$ denote the standard inner product, and $\|\cdot\|$ stand for the Euclidean norm. For any $\omega\in\rr$, let $\omega_+=\max\{\omega,0\}$ and $\lceil \omega \rceil$ denote the least integer number greater than or equal to $\omega$.

Given a proper closed convex function $h:\rr^n\rightarrow(-\infty,\infty]$, $\partial h$ denotes its subdifferential.  The proximal operator associated with $h$ is denoted by  
$\prox_h$, which is defined as
\begin{equation*}
\prox_h(z) = \argmin_{x\in\rr^n} \left\{ \frac{1}{2}\|x - z\|^2 + h(x) \right\} \qquad \forall z \in \rr^n.
\end{equation*}
Given an operator $\mT$,  $\dom{\mT}$ and $\cl(\dom{\mT})$ denote its domain and the closure of its domain, respectively. 
For a mapping $g:\rr^n \to \rr^m$, $\nabla g$ denotes the transpose of the Jacobian of $g$. The mapping $g$ is called $L$-Lipschitz continuous on a set $\Omega$ for some constant $L>0$ if $\|g(x)-g(y)\| \le L\|x-y\|$ for all $x,y\in\Omega$. Besides, $g$ is called locally Lipschitz continuous on a closed set $\widehat\Omega$ if $g$ is $L_\Omega$-Lipschitz continuous on any compact set $\Omega\subseteq\widehat\Omega$ for some $L_\Omega>0$. Let $I$ stand for the identity operator. For a maximal monotone operator $\mT:\rr^n\rightrightarrows\rr^n$, the resolvent of $\mT$ is denoted by 
$(I+\mT)^{-1}$, which is a mapping   defined everywhere in $\rr^n$. In particular, $z=(I+\mT)^{-1}(x)$ if and only if $x\in (I+\mT)(z)$. Since the evaluation of $(I+\gamma \mT)^{-1}(x)$ is often as cheap as that of $(I+\mT)^{-1}(x)$, we count the evaluation of  $(I+\gamma \mT)^{-1}(x)$ as \emph{one evaluation of resolvent} of $\mT$ for any $\gamma>0$ and $x$. The \emph{residual} of $\mT$ at a point $x\in\dom{\mT}$ is defined as $\res_\mT(x) = \inf\{\|v\|: v\in\mT(x)\}$. For any given $\varepsilon>0$, a point $x$ is called an \emph{$\varepsilon$-residual solution} of problem \mref{MI} if $x\in\dom{B}$ and $\res_{F+B}(x) \le \varepsilon$.

Given a nonempty closed convex set $C\subseteq\rr^n$, $\dist(z,C)$ stands for the Euclidean distance from $z$ to $C$, and $\Pi_C(z)$ denotes the Euclidean projection of $z$ onto $C$, namely, 
\[
 \Pi_C(z) = \argmin\{\|z - x\|: x\in C\}, \quad \dist(z,C)=\left\|z - \Pi_C(z)\right\|, \quad \forall z \in \rr^n.
\]
The normal cone of $C$ at any $z\in C$ is denoted by $\cN_C(z)$. For a closed convex cone $\mcK$, we use $\mcK^*$ to denote the dual cone of $\mcK$, that is,  $\mcK^* = \{y\in\rr^m: \langle y,x\rangle\geq 0, \; \forall x\in\mcK\}$.

\section{A primal-dual extrapolation method for problem \eqref{MI} with $\mu>0$} 
\label{PDE1}

In this section we propose a primal-dual extrapolation method for solving 
a strongly MI problem \eqref{MI}, namely, the case in which Assumption \ref{ass0}(b) holds with $\mu>0$. Our method is a variant of the classical forward-backward splitting (FBS) method \cite{lions1979splitting,passty1979ergodic}. 
It modifies the forward term in \eqref{FBS} by using a primal and dual extrapolation technique\footnote{In the context of optimization, the operator $F$ is typically the gradient of a function and $F(x)$ can be viewed as a point in the dual space. As a result, $\{x^t\}$ and $\{F(x^t)\}$ generated by this method can be respectively viewed as a primal and dual sequence and thus the extrapolations on them are called primal and dual extrapolations just for simplicity. Accordingly, we refer to our method as a \emph{primal-dual extrapolation method}.} that has recently been proposed to design algorithms for solving stochastic VI problems in \cite{huang2022new}. Note that the choice of the parameters for extrapolations in \cite{huang2022new} requires Lipschitz continuity of $F$. Since $F$ is only assumed to be locally Lipschitz continuous in this paper, the choice of them in \cite{huang2022new} is not applicable to our method. To resolve this issue, we propose a \emph{backtracking line search scheme} to decide on parameters for extrapolations and splitting.\footnote{It shall be mentioned that backtracking line search schemes have been widely used for designing algorithms for solving MI problems (e.g., see \cite{malitsky2020forward,tseng2000modified}).} In addition, we propose a \emph{verifiable termination criterion}, which guarantees that our method \emph{outputs an $\epsilon$-residual solution} of problem \eqref{MI} with $\mu>0$ for any given tolerance $\epsilon$. The proposed method is presented in \aref{ALG} below.

\begin{algorithm}[H]
\caption{A primal-dual extrapolation method for problem \eqref{MI} with $\mu>0$}
\label{ALG}
\begin{algorithmic}[1]
\REQUIRE $\epsilon>0$,  $\gamma_0>0$, $\delta \in (0,1)$, $0<\nu \le 1/2$, $\eta\in [0, \nu/(1+\nu))$, and $x^0=x^1\in\dom{B}$.
\FOR{$t=1,2,\dots$}
\STATE Compute
\begin{align}
\label{iter}
x^{t+1}=(I+\gamma_tB)^{-1}\left(x^t+\alpha_t(x^t-x^{t-1})-\gamma_t(F(x^t)+\beta_t(F(x^t)-F(x^{t-1})))\right),
\end{align}
where 
\begin{align}\label{para}
\gamma_t=\min\{\gamma_0,\delta^{-1}\gamma_{t-1}\}\delta^{n_t},\ \beta_t=\frac{\gamma_{t-1}}{\gamma_t}\left(1+\frac{2\mu\gamma_{t-1}}{1-\eta}\right)^{-1},\ \alpha_t=\frac{\eta\gamma_{t}\beta_t}{\gamma_{t-1}},
\end{align}
and $n_t$ is the smallest nonnegative integer such that
\begin{align}
\label{term}
\|F(x^{t+1})-F(x^t)-\eta\gamma_t^{-1}(x^{t+1}-x^t)\| \leq \nu(1-\eta)\gamma_t^{-1}\|x^{t+1}-x^t\|.
\end{align}
\STATE Terminate the algorithm and output $x^{t+1}$ if 
\beq\label{alg-term}
\|\gamma_t^{-1}(x^t-x^{t+1}+\alpha_t(x^t-x^{t-1}))+F(x^{t+1})-F(x^t)-\beta_t(F(x^t)-F(x^{t-1}))\|\leq\epsilon.
\eeq
\ENDFOR
\end{algorithmic}
\end{algorithm}

\begin{rem}
(i) If $\eta= 0$, \aref{ALG} is reduced to a dual extrapolation method. Besides, $\alpha_t$ and $\beta_t$ are for primal-dual extrapolation and $\gamma_t$ is the stepsize, while $\gamma_0$, $\delta$ and $\nu$ are used for backtracking line search. For the sake of generality, we provide a flexible choice for $\nu$ and $\eta$ satisfying the conditions stated in the input line of \aref{ALG}. Nevertheless, one can easily specify them, for example, letting $(\nu,\eta)=(0.5,0.33)$, which appears to be the best choice for \aref{ALG} as observed in practice.  

(ii) As will be shown in Lemma \ref{L2}, it holds that 
\[
\gamma_t^{-1}(x^t-x^{t+1}+\alpha_t(x^t-x^{t-1}))+F(x^{t+1})-F(x^t)-\beta_t(F(x^t)-F(x^{t-1}))\in (F+B)(x^{t+1}).
\]
As a result, $x^{t+1}$ satisfying \eqref{alg-term} implies that $\res_{F+B}(x^{t+1})\leq\epsilon$, namely, $x^{t+1}$ is an $\epsilon$-residual solution of problem \eqref{MI}. Thus, \eqref{alg-term} provides a verifiable termination criterion for Algorithm \ref{ALG} to find an $\epsilon$-residual solution of \eqref{MI}.

(iii) As will be established below, \aref{ALG}  is well-defined at each iteration. Moreover, one can observe that the fundamental operations of \aref{ALG}  consist only of evaluations of $F$ and resolvent of $B$. Specifically, at iteration $t$, \aref{ALG} requires $n_t+1$ evaluations of $F$ and resolvent of $B$ for finding $x^{t+1}$ satisfying \eqref{term}.
\end{rem}

We next establish that \aref{ALG} \emph{well-defined} and \emph{outputs an $\epsilon$-residual solution} of problem \eqref{MI}. We also study its complexity including: (i) \emph{iteration complexity} measured by the number of iterations; (ii) \emph{operation complexity} measured by the number of evaluations of $F$ and resolvent of $B$.

To proceed, we assume throughout this section that problem \eqref{MI} is a strongly MI problem (namely, $\mu>0$) and that $x^*$ is the solution of \eqref{MI}. Let $\{x^t\}_{t\in \bbT}$ denote all the iterates generated by \aref{ALG}, where $\bbT$ is a \emph{subset} of consecutive nonnegative integers starting from $0$.\footnote{For the time being, it is possible that $\bbT=\{0,1,2, \ldots, T\}$ or $\{0,1,2, \ldots\}$ for some $T \geq 0$. The reason for not presuming $\bbT$ to be a finite set here is that the finite termination of Algorithm \ref{ALG} is not yet established. Nevertheless, it will be shown in Theorem \ref{thm-outer} that $\bbT$ is a finite set.}  We also define 
\begin{align}
& r_0=\|x^0-x^*\|, \quad \cS=\left\{x\in\dom{B}:\ \|x-x^*\| \leq\frac{r_0}{\sqrt{1-2\nu^2}}\right\}, \label{set-S} \\
& \bbT-1 = \{t-1: t \in \bbT\},  \quad \xi = \nu(1-\eta)-\eta,   \label{xi}
\end{align}
where $x^0$ is the initial point, and $\nu$ and $\eta$ are the input parameters of Algorithm \ref{ALG}. 

The following lemma establishes that $F$ is \emph{Lipschitz} continuous on $\cS$ and also on an enlarged set induced by  $\gamma_0$, $r_0$, $\nu$, $x^*$, 
  $F$ and $\cS$, albeit $F$ is \emph{locally Lipschitz} continuous on $\cl(\dom{B})$. This result will play an important role in this section.

\begin{lemma} \label{F-Lipschtiz}
Let $\cS$ be defined in \eqref{set-S}. Then the following statements hold.
\bi
\item[(i)] $F$ is $L_\cS$-Lipschitz continuous on $\cS$ for some constant $L_\cS>0$.
\item[(ii)] $F$ is  $L_\hS$-Lipschitz continuous on $\hS$ for some constant $L_\hS>0$, where
\begin{equation}
\label{set-hS}
\hS=\left\{x\in\dom{B}:\ \|x-x^*\|\leq \frac{(5+9\gamma_0 L_\cS)r_0}{3\sqrt{1-2\nu^2}}\right\},
\footnote{The specific choices of the radius associated with $\cS$ and $\hS$ will become clear
from the proofs of Lemmas \ref{nlep} and \ref{nlemma3}.}
\end{equation}
$r_0$ is defined in \eqref{set-S}, and $\gamma_0>0$ and $\nu\in(0,1/2]$ are the input parameters of Algorithm \ref{ALG}.
 \ei
\end{lemma}

\begin{proof}
Notice that $\cS$ is a bounded subset in $\dom{B}$. By this and the local Lipschitz continuity of $F$ on $\cl(\dom{B})$, there exists some constant  $L_\cS>0$ such that $F$ is $L_\cS$-Lipschitz continuous on $\cS$. Hence, statement (i) holds and moreover the set $\hS$ is well-defined. By a similar argument, one can see that statement (ii) also holds.
\end{proof}

The following theorem shows that \aref{ALG} is well-defined at each iteration.  Its proof is deferred to Section \ref{sec:proof}.

 \begin{thm} \label{np1}
 Let $\{x^t\}_{t\in\bbT}$ and $\{n_t\}_{1 \leq t\in\bbT-1}$ be generated by \aref{ALG} and $\xi$ be defined in \eqref{xi}. Then the following statements hold. 
\bi
\item[(i)] \aref{ALG} is well-defined at each iteration.
\item[(ii)] $x^t\in\cS$ for all $t \in \bbT$, and moreover, $\sum_{i=1}^t n_i \le M+t$ for all 
$1 \le t \in \bbT-1$, where $\cS$ is defined in \eqref{set-S} and
\begin{align} \label{nlocal1}
M=\left\lceil\log\left(\frac{\xi}{\gamma_0 L_\hS}\right)/\log\delta\right\rceil_+.
\end{align}
\ei
\end{thm}

The next theorem presents iteration and operation complexity of \aref{ALG} for finding an $\epsilon$-residual solution of problem \mref{MI} with $\mu>0$, whose proof is deferred to Section \ref{sec:proof}.

\begin{thm}\label{thm-outer}
Let $\gamma_0$, $\delta$, $\nu$, $\eta$ and $\epsilon$ be given in \aref{ALG}, $L_\cS$ and $L_\hS$ be given in Lemma \ref{F-Lipschtiz}, and $r_0$ and $\xi$ be defined in \eqref{set-S} and \eqref{xi}. Suppose that $\mu>0$, i.e., $F+B$ is strongly monotone on $\dom{B}$. Then \aref{ALG} terminates and outputs an $\epsilon$-residual solution of problem \mref{MI}
in at most $T$ iterations. Moreover, the number of evaluations of $F$ and resolvent of $B$ performed in \aref{ALG} is no more than $\nn$, respectively, where
\begin{align}
\cT=& \ 3+\left\lceil2\log{\left(\frac{r_0\left(8+12\gamma_0L_\cS\right)}{3\epsilon\sqrt{1-2\nu^2}\min\left\{L_\hS^{-1}\delta\xi,\gamma_0\right\}}\right)}\middle/\log\left(1+\frac{2\mu}{1-\eta}\min\Big\{L_\hS^{-1}\delta\xi, \gamma_0\Big\}\right)\right\rceil_+, \label{t1k} \\
\nn
=& \ 2T+\left\lceil\log\left(\frac{\xi}{\gamma_0 L_\hS}\right)/\log\delta\right\rceil_+. \label{t1n}
\end{align}
\end{thm}

\begin{rem} \label{remark1}
(i) It can be seen from Theorem \ref{thm-outer} that \aref{ALG} enjoys 
an iteration and operation complexity of $\oo{\log \epsilon^{-1}}$ for finding an $\epsilon$-residual solution of problem \eqref{MI} with $\mu>0$ under the assumption that $F$ is locally Lipschitz continuous on $\cl(\dom B)$. In addition, notice that if $\gamma_0\geq \delta\xi/L_\hS$, 
\[
\log\left(1+\frac{2\mu}{1-\eta}\min\Big\{L_\hS^{-1}\delta\xi, \gamma_0\Big\}\right) \approx \frac{2\delta\xi}{1-\eta}\cdot \frac{\mu}{L_\hS}.
\]
It then follows from \eqref{t1k} and \eqref{t1n} that if $\gamma_0\geq \delta\xi/L_\hS$, $\cT$ and $\nn$ are roughly proportional to $L_\hS/\mu$. Hence, $L_\hS/\mu$ can be viewed as the ``condition number'' of problem \mref{MI} with $\mu>0$.

 (ii) Algorithm \ref{ALG} will become a linearly convergent method if setting $\epsilon=0$. Indeed, one can observe from Lemma \ref{nlep} that the sequence $\{x^k\}$ generated by Algorithm \ref{ALG} with $\epsilon=0$ satisfies $\|x^k-x^*\|^2 \leq (1-2\nu^2)^{-1}(1+2\mu\underline\gamma)^{2-k}\|x^0-x^*\|^2$ for all $k\geq 2$, where $x^*$ is the solution of \eqref{MI} and $\underline\gamma:=\inf_k \gamma_k $ is a positive number due to Theorem \ref{np1}.
\end{rem}

\section{A primal-dual extrapolation method for problem \eqref{MI} with $\mu=0$} \label{PDE2}

In this section we propose a primal-dual extrapolation method for solving a non-strongly MI problem \eqref{MI}, namely, the case in which Assumption \ref{ass0}(b) holds with $\mu=0$. Our method consists of applying \aref{ALG} to approximately solve a sequence of strongly MI problems $0\in (F_k+B)(x)$, where $F_k$ is a perturbation of $F$ given in \eqref{Fk}. The proposed method is presented in \aref{PPA}.

\begin{algorithm}[h]
\caption{A primal-dual extrapolation method for problem \eqref{MI} with $\mu=0$}
\label{PPA}
\begin{algorithmic}[1]
\REQUIRE $\varepsilon>0$, $\gamma_0>0$, $z^0\in\dom{B}$, $0<\delta<1$, $0<\nu\leq1/2$, $\eta \in [0,\nu/(1+\nu))$, $\rho_0\geq1$, $0<\tau_0\leq1$, $\zeta>1$, $0<\sigma<1/\zeta$, $\rho_k=\rho_0\zeta^k$, $\tau_k=\tau_0\sigma^k$ for all $k \ge 0$.
\FOR{$k = 0, 1, \dots$}
\STATE Call \aref{ALG} with $F \leftarrow F_k$, $\mu\leftarrow \rho_k^{-1}$, $\epsilon \leftarrow \tau_k$, $x^0=x^1 \leftarrow z^k$ and the parameters $\gamma_0$, $\eta$, $\delta$ and $\nu$, and output $z^{k+1}$, where 
\begin{align}
F_k(x)=  F(x)+\rho_k^{-1}(x-z^k) \quad \forall x \in \dom F \label{Fk}.
\end{align}
\STATE Terminate this algorithm and output $z^{k+1}$ if 
\begin{align}
\rho_k^{-1}\|z^{k+1}-z^k\| +  \tau_k\leq \varepsilon. \label{ppa-term}
\end{align}
\ENDFOR
\end{algorithmic}
\end{algorithm}

\begin{rem}
(i) In \aref{PPA}, the parameters $\gamma_0$, $\delta$, $\nu$ and $\eta$ have the same meaning as those for Algorithm \ref{ALG} (see Remark 1(i)). Besides, $\rho_0$, $\tau_0$, $\zeta$ and $\sigma$ are used for subproblem regularization and subproblem termination criterion. 

(ii) It is easy to see that \aref{PPA} is well-defined at each iteration and equipped with a verifiable termination criterion, while it shares the same fundamental operations as Algorithm \ref{ALG}, consisting only of evaluations of $F$ and resolvent of $B$. 
\end{rem}

We next show that \aref{PPA} \emph{outputs an $\varepsilon$-residual solution} of problem \eqref{MI}. We also study its complexity including: (i) \emph{iteration complexity} measured by the number of iterations; (ii) \emph{operation complexity} measured by the total number of evaluations of $F$ and resolvent of $B$.

To proceed, we assume that $x^*$ is an arbitrary solution of problem \eqref{MI} and fixed throughout this section. Let $\{z^k\}_{k\in \bbK}$ denote all the iterates generated by \aref{PPA}, where $\bbK$ is a \emph{subset} of consecutive nonnegative integers starting from $0$.\footnote{For the time being, it is possible that $\bbK=\{0,1,2, \ldots, K\}$ or $\{0,1,2, \ldots\}$ for some $K \geq 0$. The reason for not presuming $\bbK$ to be a finite set is that the finite termination of \aref{PPA} is not yet established. Nevertheless, it will be shown in Theorem \ref{thm2} that $\bbK$ is a finite set.} We also define $\bbK-1 = \{k-1: k \in \bbK\}$, and
\begin{equation}\label{set-Q}
\br_0 = \|z^0-x^*\|, \quad \cQ=\left\{x\in\dom{B}:\ \|x-x^*\|\leq\left(\frac{1}{\sqrt{1-2\nu^2}}+1\right)\left(\br_0+\frac{\rho_0\tau_0}{1-\sigma\zeta}\right)\right\},
\end{equation}
where $z^0$ is the initial point and $\rho_0$, $\tau_0$, $\nu$, $\zeta$, $\sigma$ are the input parameters of \aref{PPA}. 

The following lemma establishes that $F_k$ is Lipschitz continuous on $\cQ$ and also on an enlarged set induced by $F_k$ and $\cQ$ with a Lipschitz constant independent on $k$. 
This result will play an important role in this section.

\begin{lemma} \label{Fk-Lipschtiz}
Let $F_k$ and $\cQ$ be defined in \eqref{Fk} and \eqref{set-Q}. Then the following statements hold.
\bi
\item[(i)] $F_k$ is $L_\cQ$-Lipschitz continuous on $\cQ$ for some constant $L_\cQ>0$  independent of $k$. 
\item[(ii)] $F_k$ is  $L_\hQ$-Lipschitz continuous on $\hQ$ for some constant $L_\hQ>0$ independent of $k$, where
\begin{equation}
\label{set-hQ}
\hQ=\left\{x\in\dom{B}:\ \|x-x^*\|\leq\left(\frac{5+9\gamma_0L_\cQ}{3\sqrt{1-2\nu^2}}+1\right)\left(\br_0+\frac{\rho_0\tau_0}{1-\sigma\zeta}\right)\right\}.\footnote{The specific choices of the radius associated with $\cQ$ and $\hQ$ will become clear from the proof of Lemma \ref{belong}.}
\end{equation}
 \ei
\end{lemma}

\begin{proof}
Notice that $\cQ$ is a bounded subset in $\dom{B}$. By this and the local Lipschitz continuity of $F$ on $\cl(\dom{B})$,  there exists some constant  $\tilde L_\cQ>0$ such that $F$ is $\tilde L_\cQ$-Lipschitz continuous on $\cQ$. In addition, notice from \aref{PPA} that  $\rho_k\geq\rho_0$ for all $k\geq 0$. Using these and \eqref{Fk}, we can easily see that $F_k$ is $L_\cQ$-Lipschitz continuous on $\cQ$ with $L_\cQ=\tilde L_\cQ+1/\rho_0$. Hence, statement (i) holds and moreover the set $\hQ$ is well-defined. By a similar argument, one can see that statement (ii) also holds.
\end{proof}

The next theorem presents iteration and operation complexity of \aref{PPA} for finding an $\varepsilon$-residual solution of problem \mref{MI} with $\mu=0$, whose proof is deferred to Section \ref{sec:proof}.

\begin{thm}\label{thm2}
Let $\gamma_0$, $\delta$, $\nu$, $\eta$, $\zeta$, $\sigma$, $\rho_0$, $\tau_0$ and $\varepsilon$ be given in \aref{PPA}, $L_\cQ$ and $L_\hQ$ be given in Lemma \ref{Fk-Lipschtiz}, $\xi$ and $\br_0$  be defined in  \eqref{xi} and \eqref{set-Q}, and 
\begin{align}
\Lambda & = \frac{\rho_0\tau_0}{1-\sigma\zeta}, \quad C_1 = \log\left(\frac{(\br_0+\Lambda)\left(8+12\gamma_0L_\cQ\right)}{3\tau_0\sqrt{1-2\nu^2}\min\left\{L_\hQ^{-1}\delta\xi,\gamma_0\right\}}\right),  \label{C1}\\
C_2 &=   \left\lceil\log\left(\frac{\xi}{\gamma_0 L_\hQ}\right)/\log\delta\right\rceil_+, \quad 
C_3 =  \ \frac{1}{{(\zeta-1)\log\left(1+\frac{2}{\rho_0(1-\eta)}\min\Big\{L_\hQ^{-1}\delta\xi, \gamma_0\Big\}\right)}}. \label{C2}
\end{align}
Suppose that $\mu=0$, i.e., $F+B$ is monotone but not strongly monotone on $\dom{B}$. Then \aref{PPA} terminates and outputs an $\varepsilon$-residual  solution in at most $K+1$ iterations. Moreover, the number of evaluations of $F$ and resolvent of $B$ performed in \aref{PPA}
 is no more than $\mm$, respectively, where
\begin{align}
K = \left\lceil\max\left\{\log_\zeta\left(\frac{2\br_0+2\Lambda}{\varepsilon\rho_0}\right),\frac{\log(2\tau_0/\varepsilon)}{\log(1/\sigma)}\right\}\right\rceil_+,\label{t2N}
\end{align}
and
\begin{align}\label{complexity}
\mm =& \ 8+C_2+(8+C_2)K  +4\zeta (C_1)_+C_3
 \max\left\{\frac{2\zeta(\br_0+\Lambda)}{\varepsilon\rho_0},\zeta\left(\frac{2\tau_0}{\varepsilon}\right)^{\frac{\log\zeta}{\log(1/\sigma)}},1\right\} \notag \\
& \ +4\zeta C_3 (\log \sigma^{-1})K 
 \max\left\{\frac{2\zeta(\br_0+\Lambda)}{\varepsilon\rho_0},\zeta\left(\frac{2\tau_0}{\varepsilon}\right)^{\frac{\log\zeta}{\log(1/\sigma)}},1\right\}.
\end{align}
\end{thm}

\begin{rem} \label{remark2}
(i) Since $1<\zeta<1/\sigma$ and $K=\oo{\log \varepsilon^{-1}}$, it can be seen from Theorem \ref{thm2} that \aref{PPA} enjoys 
an iteration complexity of $\oo{\log \varepsilon^{-1}}$ and an operation complexity of $\oo{\varepsilon^{-1}\log \varepsilon^{-1}}$ for finding an $\varepsilon$-residual solution of problem \eqref{MI} with $\mu=0$ under the assumption that $F$ is locally Lipschitz continuous on $\cl(\dom B)$. The latter complexity significantly improves the previously best operation complexity $\cO(\varepsilon^{-2})$ achieved by the FRBS method \cite{malitsky2020forward}.  In addition, notice that if $\gamma_0\geq\delta\xi/L_\hQ$, 
\[
\log\left(1+\frac{2}{\rho_0(1-\eta)}\min\Big\{L_\hQ^{-1}\delta\xi, \gamma_0\Big\}\right) \approx \frac{2\delta\xi}{\rho_0(1-\eta)} L_\hQ^{-1}.
\]
It then follows from \eqref{C2} and \eqref{complexity} that if $\gamma_0\geq\delta\xi/L_\hQ$, $\mm$ is roughly proportional to $L_\hQ$. Hence, $L_\hQ$ can be viewed as the ``Lipschitz constant'' of problem \mref{MI} with $\mu=0$.

(ii) Algorithm \ref{PPA} will become a globally convergent method if setting $\varepsilon=0$. Indeed, one can observe from Lemma \ref{ppal1} that the sequence $\{z^k\}$ generated by Algorithm \ref{PPA} with $\varepsilon=0$ satisfies $\|z^k- \left(I+\rho_k(F+B)\right)^{-1}(z^k)\| \leq \rho_k\tau_k$ for all $k\geq 0$, where $0<\rho_k \to \infty$ and $\sum_k  \rho_k\tau_k < \infty$. Besides, one can see from Lemma \ref{ppal2} that $\{z^k\}$ is bounded. It then follows from \cite[Theorem 1]{Rocka76} that the sequence $\{z^k\}$ converges to a solution of \eqref{MI}.

(iii) While Algorithm \ref{PPA} is proposed to solve problem \eqref{MI} with $\mu=0$, it is also applicable to \eqref{MI} with $\mu>0$. Similar to the proof of Theorem \ref{thm2}, it can be shown that Algorithm 2 achieves an operation complexity of $\cO((\log\varepsilon^{-1})^2)$ for finding and $\varepsilon$-residual solution of problem \eqref{MI} with $\mu>0$. This complexity is at most worse by a logarithmic factor compared to the complexity achieved by directly calling Algorithm \ref{ALG}.
\end{rem}

\section{Applications} \label{app}

In this section we study applications of our PD extrapolation method, particularly \aref{PPA}, for solving several important classes of problems, particularly, convex conic optimization, conic constrained saddle point, and variational inequality problems.  As a consequence, complexity results are obtained for finding an $\varepsilon$-KKT or $\varepsilon$-residual solution of these problems under local Lipschitz continuity for the first time.

\subsection{Convex conic optimization}
In this subsection we consider convex conic optimization 
\begin{equation}\label{conic-p}
\begin{array}{rl}
\min &  f(x) + P(x) \\
\mbox{s.t.} &  -g(x)\in \mcK, 
\end{array}
\end{equation}
where $f,P:\rr^n\to(-\infty,\infty]$ are proper closed convex functions, $\mcK$ is a closed convex cone in $\rr^m$, and the mapping   $g:\rr^n\to\rr^m$ is $\mcK$-convex, that is,
\begin{align}\label{conic-convex}
\vartheta g(x) + (1-\vartheta)g(y)-g(\vartheta x + (1-\vartheta)y) \in \mcK \qquad\forall x,y\in\rr^n, \; \vartheta\in[0,1].
\end{align}
It shall be mentioned that $\dom{P}$ is possibly \emph{unbounded}.

Problem \eqref{conic-p} includes a rich class of problems as special cases. For example, when $\mathcal{K} = \rr_{+}^{m_1}\times\{0\}^{m_2}$ for some $m_1$ and $m_2$, $g(x) = (g_1(x),\ldots,g_{m_1}(x), h_1(x),\ldots,h_{m_2}(x))^T$ with convex $g_i$'s and affine $h_j$'s, and $P(x)$ is the indicator function of a simple convex set $\cX\subseteq\rr^n$, problem \eqref{conic-p} reduces to an ordinary convex optimization problem
\[
\begin{array}{rl}
\min\limits_{x \in \cX} \{f(x): g_i(x) \leq 0, \ i=1,\ldots,m_1; h_j(x) = 0, \ j=1,\ldots,m_2\}.
\end{array}
\]

We make the following additional assumptions for problem \eqref{conic-p}.
\begin{ass} \label{assump-cvx}
\begin{enumerate}[label=(\alph*)]
\item The proximal operator associated with $P$ and also the projection onto $\mcK^*$ can be exactly evaluated.
\item The function $f$ and the mapping  $g$ are differentiable on $\cl(\dom{\partial P})$. Moreover, $\nabla f$ and $\nabla g$  are locally Lipschitz continuous on $\cl(\dom{\partial P})$.
\item Both problem \mref{conic-p} and its Lagrangian dual problem
\begin{align}\label{conic-d}
\sup_{\lambda\in\mcK^*}\inf_{x}\left\{f(x)+P(x)+\langle\lambda,g(x)\rangle\right\}
\end{align}
have optimal solutions, and moreover, they share the same optimal value. 
\end{enumerate}
\end{ass}

Under the above assumptions, it can be shown that $(x,\lambda)$ is a pair of optimal solutions of \mref{conic-p} and \mref{conic-d} if and only if it satisfies the Karush-Kuhn-Tucker (KKT) condition
\begin{align} \label{KKT}
0 \in \begin{pmatrix}
\nabla f(x)+\nabla g(x)\lambda+\partial P(x)\\
-g(x)+\cN_{\mcK^*}(\lambda)
\end{pmatrix}. 
\end{align}
In general, it is difficult to find an exact optimal solution of \mref{conic-p} and \mref{conic-d}. Instead, for any given $\varepsilon>0$, we are interested in finding a pair of $\varepsilon$-KKT solutions $(x,\lambda)$ of \mref{conic-p} and \mref{conic-d} that satisfies
\begin{align} \label{eps-KKT}
\dist(0, \nabla f(x)+\nabla g(x)\lambda+\partial P(x)) \le \varepsilon, \quad 
\dist(0,-g(x)+\cN_{\mcK^*}(\lambda)) \le \varepsilon.
\end{align}

Observe from \eqref{KKT} that problems \mref{conic-p} and \mref{conic-d} can be solved as the MI problem
\begin{align}\label{conic-inclusion}
0\in F(x,\lambda)+B(x,\lambda),
\end{align}
where 
\begin{align}\label{conic-def1}
F(x,\lambda)=\begin{pmatrix}
\nabla f(x)+\nabla g(x)\lambda\\
-g(x)\end{pmatrix},\quad
B(x,\lambda)=\begin{pmatrix}
\partial P(x)\\
\cN_{\mcK^*}(\lambda)
\end{pmatrix}.
\end{align}
Notice that $\lambda\in\mcK^*$ and $g$ is $\mcK$-convex in the sense that \eqref{conic-convex} holds, which imply that $\langle\lambda,g(x)\rangle$ is convex in $x$. Based on this and the above assumptions, one can observe that $f(x)+\langle\lambda,g(x)\rangle$ is convex in $x$ and concave in $\lambda$ on $\cl(\dom B)$, which implies that $F$ is monotone on $\cl(\dom B)$.  One can also observe that $F$ is locally Lipschitz continuous on $\cl(\dom B)$ and $B$ is maximal monotone.  As a result,  \aref{PPA} can be suitably applied to the MI problem \eqref{conic-inclusion}. It then follows from Theorem \ref{thm2} that \aref{PPA}, when applied to problem \eqref{conic-inclusion}, finds an $\varepsilon$-residual solution $(x,\lambda)$ of  \eqref{conic-inclusion} within  $\oo{\varepsilon^{-1}\log{\varepsilon^{-1}}}$ evaluations of $F$ and resolvent of $B$. Notice from \eqref{eps-KKT} and \eqref{conic-def1} that such $(x,\lambda)$ is also a pair of $\varepsilon$-KKT solutions of \mref{conic-p} and \mref{conic-d}. In addition, the evaluation of $F$ requires that of $\nabla f$ and $\nabla g$, and also the resolvent of $B$ can be computed as
\begin{align*}
(I+\gamma B)^{-1}\begin{pmatrix}
x\\
\lambda\end{pmatrix}=\begin{pmatrix}
\mathrm{prox}_{\gamma P}(x)\\
\Pi_{\mcK^*}(\lambda)
\end{pmatrix} \qquad \forall (x,\lambda) \in \rr^n \times \rr^m, \gamma>0. 
\end{align*}

The above discussion leads to the following result regarding \aref{PPA} for finding a pair of $\varepsilon$-KKT solutions of problems \mref{conic-p} and \mref{conic-d}.

\begin{thm}
For any $\varepsilon>0$, \aref{PPA}, when applied to the MI problem \eqref{conic-inclusion},  outputs a pair of $\varepsilon$-KKT solutions of problems \mref{conic-p} and \mref{conic-d} within $\oo{\varepsilon^{-1}\log{\varepsilon^{-1}}}$ evaluations of $\nabla f$, $\nabla g$, $\prox_{\gamma P}$ and $\Pi_{\mcK^*}$ for some $\gamma>0$.
\end{thm}

\begin{rem}
\bi
\item[(i)] This is the first time to propose an algorithm for finding an $\varepsilon$-KKT solution of problem \mref{conic-p} without the usual assumption that $\nabla f$ and $\nabla g$ are Lipschitz continuous and/or the domain of $P$ is bounded. Moreover, the proposed algorithm is equipped with a verifiable termination criterion and enjoys an operation complexity of $\oo{\varepsilon^{-1}\log{\varepsilon^{-1}}}$. 
\item[(ii)] A first-order augmented Lagrangian method was recently proposed in \cite{lu2023iteration} for finding a pair of $\varepsilon$-KKT solutions of a subclass of problems \mref{conic-p} and \mref{conic-d}, which also requires $\oo{\varepsilon^{-1}\log{\varepsilon^{-1}}}$ evaluations of $\nabla f$, $\nabla g$, $\prox_{\gamma P}$ and $\Pi_{\mcK^*}$. However,  this method and its complexity analysis require that $\nabla f$ and $\nabla g$ be Lipschitz continuous on an open set containing $\dom{P}$ and also that $\dom{P}$ be bounded. As a result, it is generally not applicable to problem \mref{conic-p}.   
\item[(iii)] A variant of Tseng's MFBS method was proposed in \cite[Section 6]{monteiro2011complexity} for finding a pair of $\varepsilon$-KKT solutions of a special class of problems \mref{conic-p} and \mref{conic-d}, where $g$ is an affine mapping, $\mcK=\{0\}^m$, and $\nabla f$ is Lipschitz continuous on $\cl(\dom{P})$. 
Due to the latter assumption, this method is generally not applicable to problem \mref{conic-p}. Additionally, this method has an operation complexity of $\oo{\varepsilon^{-2}}$ (see  \cite[Theorem 6.3]{monteiro2011complexity}). In contrast, our method achieves a significantly better operation complexity of 
of $\oo{\varepsilon^{-1}\log{\varepsilon^{-1}}}$. Furthermore, an adaptive proximal algorithm was recently proposed in \cite[Section 3.1]{latafat2023adaptive} for solving a special case of problem \eqref{conic-p} where $g$ is an affine mapping. It has been shown in \cite[Theorem3.4]{latafat2023adaptive} that the iterates of this algorithm converges to a KKT solution of the problem.
\ei
\end{rem}

\subsection{Conic constrained saddle point problems}\label{sec:ccsp}

In this subsection we consider the following conic constrained saddle point (CCSP) problem:  
\begin{align}\label{CCSP}
\min_{-g(x)\in\mcK}\max_{-\tilde g(y)\in\tilde{\mcK}} \{\Psi(x,y) := f(x,y)+P(x)-\tilde P(y)\},
\end{align}
where $f:\rr^n\times\rr^m\to[-\infty,\infty]$ is convex in $x$ and concave in $y$, $P:\rr^n\to(-\infty,\infty]$ and $\tilde P:\rr^m\to(-\infty,\infty]$ are proper closed convex functions, $\mcK\subseteq\rr^p$ and $\tilde{\mcK}\subseteq\rr^{\tilde p}$ are closed convex cones, and 
$g$ and $\tilde g$ are $\mcK$- and $\tilde{\mcK}$-convex in the sense of \mref{conic-convex}, respectively. It shall be mentioned that $\dom{P}$ and $\dom{\tilde P}$ are possibly \emph{unbounded}.

We make the following additional assumptions for problem \eqref{CCSP}.
\begin{ass}
\begin{enumerate}[label=(\alph*)]
\item The proximal operator associated with $P$ and $\tilde P$ and also the projection onto $\mcK^*$ and $\tilde{\mcK}^*$ can be exactly evaluated. 
\item The function $f$ is differentiable on $\cl(\dom{\partial P}) \times \cl(\dom{\partial \tilde P})$. Moreover, $\nabla f$ is locally Lipschitz continuous on $\cl(\dom{\partial P}) \times \cl(\dom{\partial \tilde P})$.
\item The mappings $g$ and $\tilde g$ are respectively differentiable on $\cl(\dom{\partial P})$ and $\cl(\dom{\partial \tilde P})$. Moreover, $\nabla g$ and $\nabla \tilde g$ are locally Lipschitz continuous on $\cl(\dom{\partial P})$ and $\cl(\dom{\partial \tilde P})$, respectively. 
\item There exists a pair $(x^*,y^*)\in\dom{P}\times\dom{\tilde P}$ satisfying $-g(x^*)\in\mcK$ and $-\tilde g(y^*)\in\tilde{\mcK}$ such that
\begin{align*}
\Psi(x^*,y)\leq\Psi(x^*,y^*)\leq\Psi(x,y^*)
\end{align*}
holds for any $(x,y)\in\dom{P}\times\dom{\tilde P}$ satisfying $-g(x)\in\mcK$ and $-\tilde g(y)\in\tilde{\mcK}$.
\end{enumerate}
\end{ass}

Problem \eqref{CCSP} includes a rich class of saddle point problems as special cases. Several of them have been studied in the literature. For example,  extragradient method \cite{Kor76},  mirror-prox method \cite{Nem05}, dual extrapolation method \cite{Nest03-3}, and accelerated proximal point method \cite{Lin20} were developed for solving the special CCSP problem
\beq \label{CSP1}
\min\limits_{x\in \cX} \max\limits_{y\in\cY} \tf(x,y),
\eeq
where $\tf$ is convex in $x$ and concave in $y$ with \emph{Lipschitz continuous} gradient on $\cX \times \cY$, and $\cX$ and $\cY$ are simple convex sets. Also, optimistic gradient method \cite{Mok20-2} and extra anchored gradient method \cite{pmlr-v139-yoon21d} were proposed for solving problem \eqref{CSP1} with $\cX=\rr^n$ and $\cY=\rr^m$. In addition, accelerated proximal gradient method \cite{Tse08}, a variant of MFBS method \cite{monteiro2011complexity}, and also generalized extragradient method \cite{monteiro2011complexity} were proposed for solving the special CCSP problem
\beq \label{CSP2}
\min\limits_{x} \max\limits_{y} \left\{\tf(x,y)+P(x)-\tP(y)\right\},
\eeq
where $\tf$ is convex in $x$ and concave in $y$ with \emph{Lipschitz} continuous gradient on $\dom{P} \times \dom{\tP}$. Besides,  several optimal or nearly optimal first-order methods were developed for solving problem \eqref{CSP1} or \eqref{CSP2} with a strongly-convex-(strongly)-concave $\tf$ (e.g., see \cite{Lin20,Yang20,Tom21}). Recently, extra-gradient method
of multipliers \cite{zhang2023primal} was proposed for solving a special case of problem \mref{CCSP} with $g$ and $\tilde g$ being an affine mapping, $\dom{P}$ and $\dom{\tilde P}$ being compact, and $\nabla f$ being Lipschitz continuous on $\dom{P} \times \dom{\tilde P}$. 
Iteration complexity of these methods except \cite{pmlr-v139-yoon21d} was established based on the duality gap on the ergodic (i.e., weight-averaged) solution sequence. Yet, the duality gap can often be difficult to measure. In practice, one may use a computable upper bound on the duality gap to terminate these methods, which however typically requires the knowledge of an upper bound on the distance between the initial point and the solution set.  
Besides, there is a lack of complexity guarantees for these methods in terms of the original solution sequence. 

Due to the sophistication of the constraints $-g(x) \in \mcK$ and $-\tilde g(y) \in \tilde{\mcK}$ and also the local Lipschitz continuity of $\nabla f$, $\nabla g$ and $\nabla \tilde g$,  the aforementioned methods  \cite{Kor76,Nest03-3,Nem05,Tse08,monteiro2011complexity,Mok20-2,Lin20,Yang20,Tom21,pmlr-v139-yoon21d} are generally not suitable for solving the CCSP problem \eqref{CCSP}.  We next apply our \aref{PPA} to find an $\varepsilon$-KKT solution of \eqref{CCSP} and also study its operation complexity for finding such an approximate solution under the local Lipschitz continuity of $\nabla f$, $\nabla g$ and $\nabla \tilde g$.

Under the above assumptions, it can be shown that $(x,y)$ is a pair of optimal minmax solutions of problem \mref{CCSP} if and only if it together with some $(\lambda, \tilde \lambda)$ satisfies the KKT condition
\begin{align} \label{KKT-CSP}
0\in\begin{pmatrix}
\nabla_xf(x,y)+\nabla g(x)\lambda+\partial P(x)\\
-\nabla_yf(x,y)+\nabla \tilde g(y)\tilde{\lambda}+\partial \tilde P(y)\\
-g(x)+\cN_{\mcK^*}(\lambda)\\
-\tilde g(y)+\cN_{\tilde{\mcK}^*}(\tilde{\lambda})
\end{pmatrix}.
\end{align}
Generally, it is difficult to find a pair of exact optimal minimax solutions of \mref{CCSP}. Instead, for any given $\varepsilon>0$, we are interested in finding an $\varepsilon$-KKT solution $(x,y,\lambda,\tilde{\lambda})$ of \eqref{CCSP} that satisfies
\begin{align} 
 \dist(0, \nabla_xf(x,y)+\nabla g(x)\lambda+\partial P(x)) \le \varepsilon, & \quad 
\dist(0,-\nabla_yf(x,y)+\nabla \tilde g(y)\tilde{\lambda}+\partial \tilde P(y)) \le \varepsilon, \label{eps-KKT1}\\
 \dist(0,-g(x)+\cN_{\mcK^*}(\lambda)) \le \varepsilon, & \quad \dist(0,-\tilde g(y)+\cN_{\tilde{\mcK}^*}(\tilde{\lambda})) \le \varepsilon. \label{eps-KKT2}
\end{align}

Observe from \eqref{KKT-CSP} that problem \mref{CCSP} can be solved as the MI problem
\begin{align}\label{CSP-inclusion}
0\in F(x,y,\lambda,\tilde{\lambda})+B(x,y,\lambda,\tilde{\lambda}),
\end{align}
where 
\begin{align}\label{conic-def2}
F(x,y,\lambda,\tilde{\lambda})=\begin{pmatrix}
\nabla_xf(x,y)+\nabla g(x)\lambda\\
-\nabla_yf(x,y)+\nabla \tilde g(y)\tilde{\lambda}\\
-g(x)\\
-\tilde g(y)
\end{pmatrix},\quad
B(x,y,\lambda,\tilde{\lambda})=\begin{pmatrix}
\partial P(x)\\
\partial \tilde P(y)\\
\cN_{\mcK^*}(\lambda)\\
\cN_{\tilde{\mcK}^*}(\tilde{\lambda})
\end{pmatrix}.
\end{align}
Notice that $\lambda\in\mcK^*$, $\tilde\lambda\in\tilde\mcK^*$, and $g$ and $\tilde g$ are respectively $\mcK$- and $\tilde\mcK$-convex in the sense of \eqref{conic-convex}, which imply that $\langle\lambda,g(x)\rangle$ and $\langle\tilde\lambda,\tilde g(y)\rangle$ are convex in $x$ and $y$, respectively. Based on this and the above assumptions, one can observe that $f(x,y)+\langle\lambda,g(x)\rangle-\langle\tilde{\lambda},\tilde g(y)\rangle$ is convex in $(x,\tilde{\lambda})$ and concave in $(y,\lambda)$ on $\cl(\dom B)$, which implies that $F$ is monotone on $\cl(\dom B)$.  One can also observe that $F$ is locally Lipschitz continuous on $\cl(\dom B)$ and $B$ is maximal monotone.  As a result,  \aref{PPA} can be suitably applied to the MI problem \eqref{CSP-inclusion}. It then follows from Theorem \ref{thm2} that \aref{PPA}, when applied to problem \eqref{CSP-inclusion},  finds an $\varepsilon$-residual solution $(x,y,\lambda,\tilde{\lambda})$ of \eqref{CSP-inclusion} within $\oo{\varepsilon^{-1}\log{\varepsilon^{-1}}}$ evaluations of $F$ and resolvent of $B$. Notice from \mref{eps-KKT1}, \mref{eps-KKT2} and \eqref{conic-def2} that such $(x,y,\lambda,\tilde{\lambda})$ is also an $\varepsilon$-KKT solution of problem \eqref{CCSP}. In addition, the evaluation of $F$ requires that of $\nabla f$, $\nabla g$ and $\nabla \tilde g$, and also the resolvent of $B$ can be computed as
\begin{align*}
(I+\gamma B)^{-1}\begin{pmatrix}
x\\
y\\
\lambda\\
\tilde{\lambda}
\end{pmatrix}=\begin{pmatrix}
\mathrm{prox}_{\gamma P}(x)\\
\mathrm{prox}_{\gamma \tilde P}(y)\\
\Pi_{\mcK^*}(\lambda)\\
\Pi_{\tilde{\mcK}^*}(\tilde{\lambda})
\end{pmatrix} \qquad \forall (x,y,\lambda,\tilde{\lambda}) \in \rr^n \times \rr^m \times\rr^p\times  \rr^{\tilde p}, \gamma>0. 
\end{align*}

The above discussion leads to the following result regarding \aref{PPA} for finding an $\varepsilon$-KKT solution of problem \eqref{CCSP}.

\begin{thm}
For any $\varepsilon>0$, \aref{PPA}, when applied to the MI problem \eqref{CSP-inclusion},  outputs an $\varepsilon$-KKT solution of problem \mref{CCSP} within $\oo{\varepsilon^{-1}\log{\varepsilon^{-1}}}$ evaluations of $\nabla f$, $\nabla g$, $\nabla \tilde g$, $\prox_{\gamma P}$, $\prox_{\gamma \tilde P}$, $\Pi_{\mcK^*}$ and $\Pi_{\tilde \mcK^*}$ for some $\gamma>0$.
\end{thm}

\begin{rem}
This is the first time to propose an algorithm for finding an $\varepsilon$-KKT solution of problem \mref{CCSP}. Moreover, the proposed algorithm is equipped with a verifiable termination criterion and enjoys an operation complexity of $\oo{\varepsilon^{-1}\log{\varepsilon^{-1}}}$ without the usual assumption that $\nabla f$ is Lipschitz continuous and/or the domains $P$ and $\tilde P$ are bounded. 
\end{rem}

\subsection{Variational inequality}
In this subsection we consider the following variational inequality (VI) problem: 
\begin{align}\label{HVI}
\mbox{find }x\in\rr^n \mbox{ such that}\ g(y)-g(x)+\langle y-x,F(x)\rangle\geq0\ \forall y\in\rr^n,
\end{align}
where $g:\rr^{n}\to (-\infty,\infty]$ is a proper closed  convex function, and $F:\dom{F}\to\rr^n$ is monotone and \emph{locally Lipschitz continuous} on $\mathrm{cl}(\dom{\partial g})\subseteq\dom{F}$. It shall be mentioned that $\dom{g}$ is possibly \emph{unbounded}. Assume that problem \eqref{HVI} has at least one solution. 
For the details of VI and its applications, we refer the reader to \cite{FaPa07} and
the references therein.

Recently, an adaptive golden ratio algorithm was proposed in \cite[Algorithm 1]{malitsky2020golden} for solving \eqref{HVI}. In addition, some special cases of \mref{HVI} have been well studied in the literature. For example,  projection method \cite{Sib70}, extragradient method \cite{Kor76},  mirror-prox method \cite{Nem05}, dual extrapolation method \cite{Nest03-3},  operator extrapolation method \cite{kotsalis2020simple}, extra-point method \cite{huang2021unifying, huang2022new}, and extra-momentum method \cite{huang2022new} were developed for solving problem \mref{HVI} with $g$ being the indicator function of a closed convex set and $F$ being \emph{Lipschitz continuous} on it or the entire space. In addition, a variant of Tseng's MFBS method \cite{monteiro2011complexity}, and generalized extragradient method \cite{monteiro2011complexity} were proposed for solving 
problem \mref{HVI} with $F$ being \emph{Lipschitz continuous}. 
Iteration complexity of these methods except \cite{kotsalis2020simple} was established based on the weak gap or its variant on the ergodic (i.e., weight-averaged) solution sequence. Yet, the weak gap can often be difficult to measure. In practice, one may use a computable upper bound on the weak gap to terminate these methods, which however typically requires the knowledge of an upper bound on the distance between the initial point and the solution set.  
Besides, there is a lack of complexity guarantees for these methods in terms of the original solution sequence. In addition, since $F$ is only assumed to be \emph{locally Lipschitz continuous} on $\mathrm{cl}(\dom{g})$ in our paper, these methods are generally not suitable for solving problem \mref{HVI}. 

Generally, it is difficult to find an exact solution of problem \mref{HVI}. Instead, for any given $\varepsilon>0$, we are interested in finding an $\varepsilon$-residual solution of \eqref{HVI}, which is a point $x$ satisfying $\res_{F+\partial g}(x) \le \varepsilon$.
To this end, we first observe that problem \mref{HVI} is equivalent to the MI problem
\begin{align}\label{HVI-inclusion}
0\in (F+\partial g)(x).
\end{align}
Since $F$ is monotone and locally Lipschitz continuous on $\cl(\dom{\partial g})$ and $\partial g$ is maximal monotone,  \aref{PPA} can be suitably applied to the MI problem \eqref{HVI-inclusion}. It then follows from Theorem \ref{thm2} that \aref{PPA}, when applied to problem \eqref{HVI-inclusion},  finds an $\varepsilon$-residual solution $x$ of \eqref{HVI-inclusion}, which is indeed also an $\varepsilon$-residual solution of \eqref{HVI}, within $\oo{\varepsilon^{-1}\log{\varepsilon^{-1}}}$ evaluations of $F$ and resolvent of $\partial g$. Notice that the resolvent of $\partial g$ can be computed as
\begin{align*}
(I+\gamma \partial g)^{-1}(x)=\mathrm{prox}_{\gamma g}(x),\quad\forall x\in\rr^n,\gamma>0.
\end{align*}

The above discussion leads to the following result regarding \aref{PPA} for finding an $\varepsilon$-residual solution of problem \eqref{HVI}.

\begin{thm}
For any $\varepsilon>0$,  \aref{PPA}, when applied to the MI problem \eqref{HVI-inclusion},  outputs an $\varepsilon$-residual solution of problem \mref{HVI} within $\oo{\varepsilon^{-1}\log{\varepsilon^{-1}}}$ evaluations of $F$ and $\prox_{\gamma g}$ for some $\gamma>0$.
\end{thm}

\begin{rem}
An adaptive golden ratio algorithm was recently proposed in \cite[Algorithm 1]{malitsky2020golden}. While \cite{malitsky2020golden} did not specifically study the operation complexity of this algorithm for finding an $\varepsilon$-residual solution of \eqref{HVI}, it can be shown that the algorithm achieves an operation complexity of $\cO(\varepsilon^{-2})$ for such a solution by using \cite[equation (34) and Lemma 2]{malitsky2020golden}. In contrast, the operation complexity of $\oo{\varepsilon^{-1}\log{\varepsilon^{-1}}}$ achieved by our method is significantly better.
\end{rem}

\section{Numerical results}\label{sec:exp}

In this section we conduct some preliminary experiments to test the performance of our proposed method (Algorithm \ref{PPA}), and compare it with FRBS method \cite{malitsky2020forward}, MFBS method with an Armijo-Goldstein-type stepsize \cite{tseng2000modified}, and adaptive golden ratio (AGR) algorithm \cite{malitsky2020golden}, respectively. All the methods are coded in Matlab and all the computations are performed on a desktop with a 3.60 GHz Intel i7-12700K 12-core processor and 32 GB of RAM.

We consider the problem
\begin{equation}\label{l4-prob}
\min_{x\geq0} \max_{\|y\| \leq 1} \Big\{ \|Ax-b\|_4^4+\langle Bx, y\rangle-\|Cy-d\|^4_4\Big\},
\end{equation}
where $A\in\rr^{l\times n}$, $B\in\rr^{m\times n}$, $C\in\rr^{q\times m}$, $b\in\rr^l$, $d\in\rr^q$, and $\|z\|_4=(\sum_i z_i^4)^{1/4}$ for any vector $z$. 

We randomly generate instances for problem \eqref{l4-prob}. Specifically, we first randomly generate $U\in\rr^{l\times (n/10)}$ and $V\in\rr^{(n/10)\times n}$ with all the entries independently chosen from a normal distribution with mean $0$ and standard deviation $0.1$, and a diagonal matrix $D\in\rr^{(n/10)\times (n/10)}$ with all the diagonal entries independently chosen from a uniform distribution between 0 and 1. Then we set $A=UDV$. In a similar vein, we randomly generate $C$.  Besides, we randomly generate $P\in\rr^{m\times l}$ with all the entries independently chosen from the standard normal distribution, and set $B=PA$. In addition, we randomly generate $b$ and $d$ with all the entries independently chosen from the standard normal distribution.

Notice that problem \eqref{l4-prob} is a special case of \eqref{CCSP}. As discussed in Subsection \ref{sec:ccsp},  \eqref{l4-prob} is equivalent to the monotone inclusion problem \eqref{MI} with
\begin{align*}
F(x,y)=\begin{pmatrix}
\nabla (\|Ax-b\|_4^4)+B^Ty\\ \nabla (\|Cy-d\|_4^4)-Bx
\end{pmatrix},\qquad B(x,y)=\begin{pmatrix}
\cN_{\rr^n_+}(x)\\
\cN_{\mathcal{B}}(y)
\end{pmatrix},
\end{align*}
where $\mathcal{B}=\{z\in\rr^m:\|z\|\leq1\}$. Clearly, $B$ is a maximal monotone operator and $F$ is monotone and locally Lipschitz continuous, albeit {\it not globally} Lipschitz continuous. In addition, for any $\gamma>0$, the resolvent of $\gamma B$ can be calculated as
\[
(I+\gamma B)^{-1}\begin{pmatrix}
x\\
y\end{pmatrix}=\begin{pmatrix}
\Pi_{\rr^n_+}(x)\\
\Pi_{\mathcal{B}}(y)
\end{pmatrix}.
\]
As a result,  \eqref{l4-prob} can be suitably solved by Algorithm \ref{PPA},  FRBS \cite{malitsky2020forward}, MFBS  \cite{tseng2000modified}, and AGR \cite{malitsky2020golden}. Our aim is to find a $10^{-4}$-residual solution of the corresponding monotone inclusion problem of  \eqref{l4-prob} for the above instances by using Algorithm \ref{PPA}, FRBS, MFBS and AGR, and compare their performance. Due to this, we terminate them once a $10^{-4}$-residual solution is found.  In addition, for all the methods, we choose $0$ as the initial point and set the parameters as
\begin{itemize}
\item $(\varepsilon,\gamma_0,\delta,\nu,\rho_0,\tau_0,\zeta,\sigma,\eta)=(10^{-4},0.1,0.9,0.5,10,0.09,9,0.1,0.33)$ for Algorithm \ref{PPA};
\item $(\lambda_0,\delta,\sigma)=(0.1,0.5,0.9)$ for FRBS \cite{malitsky2020forward};
\item $(\sigma,\theta,\beta)=(0.1,0.5,0.9)$ for MFBS \cite{tseng2000modified}.
\item $(\lambda_0,\bar\lambda,\phi)=(1,1,1.5)$ for AGR \cite{malitsky2020golden}.
\end{itemize}

The computational results of Algorithm \ref{PPA}, FRBS, MFBS and AGR for the instances randomly generated above are presented in Table \ref{l4-table}. In detail, the value of $(n, m ,l , q)$ is listed in the first four columns. For each instance, the number of gradient evaluations and the CPU time (in seconds) are given in the rest of the columns. One can observe that our method, namely Algorithm \ref{PPA}, substantially outperforms the other three methods in terms of number of gradient evaluations and CPU time. Notice that our method uses both primal and dual extrapolation schemes, while FRBS only uses a dual extrapolation scheme, and MFBS and AGR do not use any of them. The numerical results in Table \ref{l4-table} demonstrate that primal and dual extrapolation schemes have an acceleration effect.
\begin{table}[H]
\centering
\resizebox{0.95\linewidth}{!}{
\begin{tabular}{cccc||llll||llll}
\hline
&&&&\multicolumn{4}{c||}{Gradient evaluations}&\multicolumn{4}{c}{CPU time (seconds)}\\
$n$&$m$&$l$&$q$&Algorithm \ref{PPA}&FRBS&MFBS&AGR&Algorithm \ref{PPA}&FRBS&MFBS&AGR\\\hline
100&10&500&100&$1.23\times10^3$&$3.12\times10^3$&$3.08\times10^3$&$2.12\times10^3$&0.3&0.5&0.4&0.3\\
200&20&1000&200&$2.81\times10^3$&$8.15\times10^3$&$1.20\times10^4$&$7.41\times10^3$&1.4&6.2&6.3&3.7\\
300&30&1500&300&$2.28\times10^4$&$6.25\times10^4$&$8.84\times10^4$&$4.27\times10^4$&23.3&82.3&114.3&57.3\\
400&40&2000&400&$7.62\times10^4$&$1.78\times10^5$&$3.84\times10^5$&$1.24\times10^5$&109.4&314.4&598.4&231.9\\
500&50&2500&500&$5.93\times10^4$&$1.68\times10^5$&$5.95\times10^5$&$1.65\times10^5$&149.2&388.5&1741.8&464.4\\
600&60&3000&600&$6.00\times10^4$&$1.70\times10^5$&$4.28\times10^5$&$1.71\times10^5$&238.4&549.0&1588.8&561.7\\
700&70&3500&700&$6.90\times10^4$&$1.54\times10^5$&$4.71\times10^5$&$1.37\times10^5$&268.3&639.5&2005.9&596.7\\
800&80&4000&800&$4.62\times10^4$&$8.52\times10^4$&$4.04\times10^5$&$8.52\times10^4$&271.6&565.9&2363.5&577.5\\
900&90&4500&900&$5.43\times10^4$&$9.33\times10^4$&$5.17\times10^5$&$8.27\times10^4$&324.3&594.3&3459.8&562.8\\
1000&100&5000&1000&$3.32\times10^4$&$6.13\times10^4$&$5.37\times10^5$&$6.67\times10^4$&380.5&784.2&6206.5&854.8\\
\hline
\end{tabular}
}
\caption{Numerical results for problem \eqref{l4-prob}}\label{l4-table}
\end{table}

\section{Proof of the main results} \label{sec:proof}

In this section we provide a proof of our main results presented in Sections \ref{PDE1} and \ref{PDE2}, which are particularly Theorems \ref{np1}, \ref{thm-outer}, and \ref{thm2}.

\subsection{Proof of the main results in Section \ref{PDE1}}
\label{sec:pf-PDE1}

In this subsection we first establish several technical lemmas and then use them to prove Theorems \ref{np1} and \ref{thm-outer}.

Before proceeding, we introduce some notation that will be used shortly. Recall from Section \ref{PDE1} that $\{x^t\}_{t\in \bbT}$ denotes all the iterates generated by \aref{ALG}, where $\bbT$ is a subset of consecutive nonnegative integers starting from $0$. For any $1\le t \in \bbT$, we define
\begin{align}
\label{def-D}
\Delta^t&= F(x^t)-F(x^{t-1}),\\
\label{def-tD}
\td^t&= \Delta^t-\eta\gamma_{t-1}^{-1}(x^{t}-x^{t-1}).
\end{align}
In addition, we define
\begin{align}
\label{def-v}
v^t&=\gamma_t^{-1}(x^t-x^{t+1}+\alpha_t(x^t-x^{t-1})+\gamma_t\Delta^{t+1}-\gamma_t\beta_t\Delta^t) \qquad \forall 1 \le t \in \bbT-1, \\
\label{def-tg}
\tg_t&=\frac{\gamma_t}{1-\eta},\qquad 
\theta_t=\prod_{i=1}^{t-1}\left(1+2\mu\tg_i\right)=\prod_{i=1}^{t-1}\left(1+\frac{2\mu\gamma_{i}}{1-\eta}\right) \quad \forall 1 \le t \in \bbT-1.\footnotemark
\end{align}
\footnotetext{We set $\theta_1=1$.}

The following lemma establishes some properties of $\{v^t\}_{1\leq t\in\bbT-1}$.

\begin{lemma}\label{L2}
Let $\{x^t\}_{t\in\bbT}$ be generated by \aref{ALG}. Then for all $1 \le t \in \bbT-1$, the following relations hold. 
\begin{align}
v^t &\in(F+B)(x^{t+1}), \label{inclusion} \\
v^t &=\tg_t^{-1}(x^t-x^{t+1}+\tg_t\td^{t+1}-\tg_t\beta_t\td^t). \label{new-def-v}
\end{align}
\end{lemma}

\begin{proof}
By \mref{iter}, one has
\[
x^t+\alpha_t(x^t-x^{t-1})-\gamma_t(F(x^t)+\beta_t(F(x^t)-F(x^{t-1})))\in x^{t+1}+\gamma_tB(x^{t+1}).
\]
Adding $\gamma_tF(x^{t+1})$ to both sides of this relation, we obtain
\[
x^t+\alpha_t(x^t-x^{t-1})+\gamma_t(F(x^{t+1})-F(x^t))-\gamma_t\beta_t(F(x^t)-F(x^{t-1}))\in x^{t+1}+\gamma_t(F+B)(x^{t+1}),
\]
which together with \eqref{def-D} and  \mref{def-v} yields
\[
v^t=\gamma_t^{-1}(x^t-x^{t+1}+\alpha_t(x^t-x^{t-1})+\gamma_t\Delta^{t+1}-\gamma_t\beta_t\Delta^t)\in(F+B)(x^{t+1}),
\]
and hence \eqref{inclusion} holds. In addition, recall from \eqref{para} that $\alpha_t=\eta\gamma_t\beta_t/\gamma_{t-1}$. By this, \mref{def-tD} and \mref{def-tg}, one has
\begin{align*}
v^t=&\ \gamma_t^{-1}(x^t-x^{t+1}+\alpha_t(x^t-x^{t-1})+\gamma_t\Delta^{t+1}-\gamma_t\beta_t\Delta^t)\notag\\
=&\ \frac{1}{\gamma_t}\left((1-\eta)(x^t-x^{t+1})+\gamma_t(\Delta^{t+1}-\frac{\eta}{\gamma_t}(x^{t+1}-x^{t}))-\gamma_t\beta_t(\Delta^t-\frac{\alpha_t}{\gamma_t\beta_t}(x^t-x^{t-1}))\right)\notag\\
=&\ \tg_t^{-1}(x^t-x^{t+1}+\tg_t\td^{t+1}-\tg_t\beta_t\td^t).
\end{align*}
Hence, \eqref{new-def-v} holds as desired.
\end{proof}

The next two lemmas establish some properties of $\{x^t\}_{t\in\bbT}$.

\begin{lemma} \label{nl2}
Let $\{x^t\}_{t\in\bbT}$ be generated by \aref{ALG}. 
Then for all $1 \le k \in \bbT-1$, we have
\begin{align}
\label{l3a}
\frac{1}{2}\theta_1\|x^0-x^*\|^2-\frac{1}{2}(1+2\mu\tg_{k})\theta_{k}\|x^{k+1}-x^*\|^2\geq-\tg_{k}\theta_{k}\langle\td^{k+1},x^{k+1}-x^*\rangle+R_k,
\end{align}
where
\begin{equation}
R_k=\sum_{t=1}^k\left(\tg_t\beta_t\theta_t\langle\td^t,x^{t+1}-x^t\rangle+\frac{1}{2}\theta_t\|x^{t+1}-x^t\|^2\right). \label{Qk}
\end{equation}
\end{lemma}

\begin{proof}
By \mref{monotone}, \mref{inclusion} and $0 \in (F+B)(x^*)$, one has
\begin{align*}
\langle v^t,x^{t+1}-x^*\rangle\geq\mu\|x^{t+1}-x^*\|^2,
\end{align*}
which along with \mref{new-def-v} implies that
\begin{align*}
\tg_t\mu\|x^{t+1}-x^*\|^2\leq& \ \langle  x^t-x^{t+1}+\tg_t\td^{t+1}-\tg_t\beta_t\td^t,x^{t+1}-x^*\rangle\\
=&\ \langle x^t-x^{t+1},x^{t+1}-x^*\rangle+\tg_t\langle\td^{t+1},x^{t+1}-x^*\rangle-\tg_t\beta_t\langle\td^t,x^{t+1}-x^*\rangle\\
=&\ \frac{1}{2}\left(\|x^t-x^*\|^2-\|x^{t+1}-x^*\|^2-\|x^t-x^{t+1}\|^2\right)+\tg_t\langle\td^{t+1},x^{t+1}-x^*\rangle\\
&-\tg_t\beta_t\langle\td^{t},x^{t}-x^*\rangle-\tg_t\beta_t\langle\td^t,x^{t+1}-x^t\rangle.
\end{align*}
Rearranging the terms in the above inequality yields
\begin{align*}
\frac{1}{2}\|x^t-x^*\|^2-\frac{1}{2}(1+2\tg_t\mu)\|x^{t+1}-x^*\|^2\geq&\  \tg_t\beta_t\langle\td^t,x^t-x^*\rangle-\tg_t\langle\td^{t+1},x^{t+1}-x^*\rangle\\
&\ +\tg_t\beta_t\langle\td^t,x^{t+1}-x^t\rangle+\frac{1}{2}\|x^{t+1}-x^t\|^2. 
\end{align*}
Multiplying both sides of this inequality by $\theta_t$ and summing it up for $t=1,\dots,k$, we have
\begin{align}
& \sum_{t=1}^{k}\left(\frac{1}{2}\theta_t\|x^t-x^*\|^2-\frac{1}{2}(1+2\tg_t\mu)\theta_t\|x^{t+1}-x^*\|^2\right) \notag  \\
& \geq \sum_{t=1}^k\theta_t\tg_t\beta_t\langle\td^t,x^t-x^*\rangle -\sum_{t=1}^k\theta_t\tg_t\langle\td^{t+1},x^{t+1}-x^*\rangle+R_k\notag \\
&=\theta_1\tg_1\beta_1\langle\td^1,x^1-x^*\rangle+\sum_{t=1}^{k-1}(\theta_{t+1}\tg_{t+1}\beta_{t+1}-\theta_t\tg_t)\langle\td^{t+1},x^{t+1}-x^*\rangle-\tg_{k}\theta_{k}\langle\td^{k+1},x^{k+1}-x^*\rangle+R_k. \label{theta4}
\end{align}
In addition, it follows from $x^0=x^1$ and \eqref{def-tD} that $\td^1=0$. Also, by the definition of $\tg_t$, $\theta_t$ and $\beta_t$ in \mref{para} and \mref{def-tg},  one has 
\begin{align}
\theta_{t+1}\tg_{t+1}\beta_{t+1}-\theta_t\tg_t&=\theta_{t}\left(1+\frac{2\mu\gamma_{t}}{1-\eta}\right)\cdot\frac{\gamma_{t+1}}{1-\eta}\cdot\frac{\gamma_{t}}{\gamma_{t+1}}\left(1+\frac{2\mu\gamma_{t}}{1-\eta}\right)^{-1}-\theta_{t}\frac{\gamma_t}{1-\eta}=0, \notag \\ 
\theta_{t+1}&=(1+2\tg_t\mu)\theta_t. \label{theta-t}
\end{align}
Using these, $\td^1=0$ and \mref{theta4}, we obtain
\begin{align*}
\sum_{t=1}^{k}\left(\frac{1}{2}\theta_t\|x^t-x^*\|^2-\frac{1}{2}\theta_{t+1}\|x^{t+1}-x^*\|^2\right)
\geq-\tg_{k}\theta_{k}\langle\td^{k+1},x^{k+1}-x^*\rangle+R_k,
\end{align*}
which yields
\begin{align*}
\frac{1}{2}\theta_1\|x^0-x^*\|^2-\frac{1}{2}\theta_{k+1}\|x^{k+1}-x^*\|^2\geq-\tg_{k}\theta_{k}\langle\td^{k+1},x^{k+1}-x^*\rangle+R_k.
\end{align*}
The conclusion then follows from this and \eqref{theta-t} with $t=k$.
\end{proof}

\begin{lemma}\label{nlep}
Let $\{x^t\}_{t\in\bbT}$ be generated by \aref{ALG}.  Then we have
\begin{equation}
\|x^{k+1}-x^*\|^2\leq\frac{1}{(1-2\nu^2)\theta_{k}}\|x^0-x^*\|^2 \qquad \forall 1 \le k \in \bbT-1. 
\label{l3b}
\end{equation}
\end{lemma}

\begin{proof}
By the definition of $\beta_t$ and $\tg_t$ in \mref{para} and \mref{def-tg},  one has $\tg_{t-1}^{-1}\tg_t\beta_t=\left(1+2\mu\gamma_{t-1}/(1-\eta)\right)^{-1}$. Using this and the definition of $\theta_t$ in \mref{def-tg}, we obtain
\begin{align}
\label{theta3}
\theta_{t-1}-4\nu^2\tg_{t-1}^{-2}\tg_t^2\beta_t^2\theta_t=& \ \theta_{t-1}\left(1-4\nu^2\left(1+\frac{2\mu\gamma_{t-1}}{1-\eta}\right)^{-2}\frac{\theta_{t}}{\theta_{t-1}}\right)\notag\\
\overset{\mref{def-tg}}{=}& \ \theta_{t-1}\left(1-4\nu^2\left(1+\frac{2\mu\gamma_{t-1}}{1-\eta}\right)^{-1}\right)\geq0,
\end{align}
where the last inequality follows from the fact that $0<\nu\leq1/2$. In addition, it follows from \mref{term}, \eqref{def-D}, \eqref{def-tD}, and the definition of $\tg_t$ in \eqref{def-tg} that
\begin{align}
\label{term-var}
\|\td^t\|\leq\nu\tg_{t-1}^{-1} \|x^t-x^{t-1}\| \qquad \forall 2 \le t \in \bbT.
\end{align}
Recall that $R_k$ is defined in \eqref{Qk}. Letting $\theta_0=0$, and using \eqref{Qk}, \eqref{theta3}, \eqref{term-var} and $x^0=x^1$, we have
\begin{align}
R_k\overset{\eqref{Qk}}{\geq}& \ \sum_{t=1}^k\left(-\tg_t\beta_t\theta_t\|\td^t\|\|x^{t+1}-x^t\|+\frac{1}{2}\theta_t\|x^{t+1}-x^t\|^2\right)\notag\\
\overset{\eqref{term-var}}{\ge}& \ \sum_{t=1}^k\left(-\nu\tg_{t-1}^{-1}\tg_t\beta_t\theta_t\|x^t-x^{t-1}\|\|x^{t+1}-x^t\|+\frac{1}{2}\theta_t\|x^{t+1}-x^t\|^2\right)\notag\\
=& \ \sum_{t=1}^k\left(-\nu\tg_{t-1}^{-1}\tg_t\beta_t\theta_t\|x^t-x^{t-1}\| \|x^{t+1}-x^t\|+\frac{1}{4}\theta_t\|x^{t+1}-x^t\|^2  +\frac{1}{4}\theta_{t-1}\|x^{t}-x^{t-1}\|^2 \right) \notag\\ 
&\,\, +\frac{1}{4}\theta_k\|x^{k+1}-x^k\|^2\notag\\
\ge & \ \sum_{t=1}^k\left(\left(\sqrt{\theta_t\theta_{t-1}}/2-\nu\tg_{t-1}^{-1}\tg_t\beta_t\theta_t\right)\|x^t-x^{t-1}\|\|x^{t+1}-x^t\|\right)  +\frac{1}{4}\theta_k\|x^{k+1}-x^k\|^2 \notag\\
\overset{\mref{theta3}}{\geq} & \ \frac{1}{4}\theta_k\|x^{k+1}-x^k\|^2. \notag
\end{align}
Using this, \eqref{l3a} and \eqref{term-var}, we further obtain
\begin{align*}
\frac{1}{2}\theta_1\|x^0-x^*\|^2-\frac{1}{2}(1+2\mu\tg_{k})\theta_{k}\|x^{k+1}-x^*\|^2\geq&-\tg_{k}\theta_{k}\langle\td^{k+1},x^{k+1}-x^*\rangle+\frac{1}{4}\theta_k\|x^{k+1}-x^k\|^2\\
\geq &-\tg_{k}\theta_{k}\|\td^{k+1}\|\|x^{k+1}-x^*\|+\frac{1}{4}\theta_k\|x^{k+1}-x^k\|^2\\
\overset{\eqref{term-var}}{\geq}&-\nu\theta_k\|x^{k+1}-x^k\| \|x^{k+1}-x^*\|+\frac{1}{4}\theta_k\|x^{k+1}-x^k\|^2\\
\geq&-\nu^2\theta_k\|x^{k+1}-x^*\|^2.
\end{align*}
It then follows from this, $\theta_1=1$, and $0<\nu\leq1/2$ that 
\begin{align*}
\|x^{k+1}-x^*\|^2\leq\frac{\theta_1}{(1+2\mu\tg_{k}-2\nu^2)\theta_k}\|x^0-x^*\|^2\leq\frac{1}{(1-2\nu^2)\theta_k}\|x^0-x^*\|^2.
\end{align*}
\end{proof}

In what follows, we will show that $\{n_t\}_{1\leq t\in \bbT-1}$ is bounded, that is, the number of evaluations of $F$ and resolvent of $B$ is bounded above by a constant for all iterations $t\in \bbT-1$.  To this end, we define 
\begin{align}
&x^{t+1}(\gamma)=(I+\gamma B)^{-1}\left(x^t+\alpha_t(\gamma)(x^t-x^{t-1})-\gamma \left(F(x^t)+\beta_t(\gamma)(F(x^t)-F(x^{t-1}))\right)\right) \quad \forall \gamma >0, \label{def-x} \\
&V_{t+1}(\gamma)= \|F(x^{t+1}(\gamma))-F(x^{t})- \eta \gamma^{-1}(x^{t+1}(\gamma)-x^t)\| \quad \forall \gamma >0, \label{L-gamma} 
\end{align}
where 
\beq \label{lambda-gamma}
\beta_t(\gamma)=\frac{\gamma_{t-1}}{\gamma}\left(1+\frac{2\mu\gamma_{t-1}}{1-\eta}\right)^{-1}, \quad \alpha_t(\gamma)=\frac{\eta\gamma\beta_t(\gamma)}{\gamma_{t-1}}. 
\eeq

The following lemma establishes some property of $x^{t+1}(\gamma)$, which will be used shortly.

\begin{lemma} \label{nlemma3}
Let $\cS$ and $\hS$ be defined in \eqref{set-S} and \eqref{set-hS}.  Assume that $x^t,x^{t-1}\in\cS$ for some $1 \le t \in \bbT-1$. Then $x^{t+1}(\gamma)\in \hS$ for any $0<\gamma \le \gamma_0$.
\end{lemma}

\begin{proof}
Fix any $\gamma \in (0,\gamma_0]$.  It follows from \mref{def-x} that
\begin{align*}
x^t-x^{t+1}(\gamma)+\alpha_t(\gamma)(x^t-x^{t-1})-\gamma(F(x^t)+\beta_t(\gamma)(F(x^t)-F(x^{t-1})))\in\gamma B(x^{t+1}(\gamma)).
\end{align*}
Also, by the definition of $x^*$, one has $-\gamma F(x^* )\in \gamma B(x^*)$. These along with the monotonicity of $B$ imply that
\begin{align}
\langle x^t-x^{t+1}(\gamma)+w,x^{t+1}(\gamma)-x^*\rangle\geq0, \label{curve-cond}
\end{align}
where 
\beq \label{w}
w=\alpha_t(\gamma)(x^t-x^{t-1})-\gamma(F(x^t)-F(x^*))-\gamma\beta_t(\gamma)(F(x^t)-F(x^{t-1})).
\eeq 
It follows from \eqref{curve-cond} that 
\begin{align*}
\|x^{t+1}(\gamma)-x^*\|^2 \le \langle x^t-x^*+w, x^{t+1}(\gamma)-x^* \rangle \le \|x^t-x^*+w\|
\|x^{t+1}(\gamma)-x^*\|,
\end{align*}
which implies that
\beq \label{eq-xg}
\|x^{t+1}(\gamma)-x^*\| \le \|x^t-x^*+w\| \le \|x^t-x^*\|+\|w\|.
\eeq
Notice from \aref{ALG} that $0<\gamma_{t-1}\leq\gamma_0$ and $0\leq\eta<1/3$. Using these and \eqref{lambda-gamma}, we have
\begin{align}
&\gamma\beta_t(\gamma)=\gamma_{t-1}\left(1+\frac{2\mu\gamma_{t-1}}{1-\eta}\right)^{-1}\leq\gamma_0, \label{lamb-bnd}\\ &\alpha_t(\gamma)=\frac{\eta\gamma\beta_t(\gamma)}{\gamma_{t-1}}=\eta\left(1+\frac{2\mu\gamma_{t-1}}{1-\eta}\right)^{-1}\leq\frac{1}{3}.  \label{gamma-bnd}
\end{align}
Recall that $\cS$, $r_0$ and $w$ are given in \mref{set-S} and \eqref{w}, respectively. Using    $x^t,x^{t-1},x^*\in\cS$, $0<\gamma\leq \gamma_0$, \mref{set-S}, \eqref{w}, \eqref{lamb-bnd} and \eqref{gamma-bnd}, we have 
\begin{align*}
\|w\|\leq& \ \alpha_t(\gamma)\|x^t-x^{t-1}\|+\gamma L_\cS\|x^t-x^*\|+\gamma\beta_t(\gamma) L_\cS\|x^t-x^{t-1}\|\\
\leq& \ \left(\frac{1}{3}+\gamma_0L_\cS\right)\|x^t-x^{t-1}\|+\gamma_0L_\cS\|x^t-x^*\|\\
\leq&\ \left(\frac{1}{3}+\gamma_0L_\cS\right)(\|x^t-x^*\|+\|x^{t-1}-x^*\|)+\gamma_0L_\cS\|x^t-x^*\|\\
\leq&\ \frac{2}{\sqrt{1-2\nu^2}}\left(\frac{1}{3}+\gamma_0L_\cS\right)r_0+\frac{1}{\sqrt{1-2\nu^2}}\gamma_0L_\cS r_0 \leq \frac{(2+9\gamma_0 L_\cS)r_0}{3\sqrt{1-2\nu^2}}.
\end{align*}
This together with $x^t \in \cS$, \mref{set-S} and \mref{eq-xg} yields
\begin{align*}
\|x^{t+1}(\gamma)-x^*\| \leq& \|x^t-x^*\| + \|w\| < \frac{(5+9\gamma_0 L_\cS)r_0}{3\sqrt{1-2\nu^2}}.
\end{align*}
The conclusion then follows from this and the definition of $\hS$ in \mref{set-hS}.
\end{proof}

The next lemma provides an upper bound on $n_t$, which will be used to prove Theorem \ref{np1}. 

\begin{lemma} \label{nlemma4}
Assume that $x^{t-1},x^t \in\cS$ for  $t \ge 1$ and  \aref{ALG} has not yet terminated at   iteration $t-1$. 
Then $x^{t+1}$ is successfully generated by \aref{ALG} at iteration $t$ with $n_t \leq M+t-\sum_{i=1}^{t-1} n_i$, where $M$ is given in \eqref{nlocal1}.
\end{lemma}

\begin{proof}
Recall that $x^{t+1}(\gamma)$ and $V_{t+1}(\gamma)$ are defined in \eqref{def-x} and 
\eqref{L-gamma}, respectively. It follows from \lref{nlemma3} that $x^{t+1}(\gamma)\in \hS$ for any $0<\gamma\leq\gamma_0$. Also, notice that $x^t \in \cS \subset \hS$. By these and Lemma \ref{F-Lipschtiz}(ii), one has 
\[
\|F(x^{t+1}(\gamma))-F(x^{t})\| \le L_\hS \|x^{t+1}(\gamma)-x^t\| \quad \forall 0<\gamma\leq\gamma_0.
\]
Using this and \eqref{L-gamma}, we obtain that for any $0<\gamma\leq\gamma_0$, 
\begin{align}
V_{t+1}(\gamma)\leq \|F(x^{t+1}(\gamma))-F(x^{t})\|+ \eta\gamma^{-1}\|x^{t+1}(\gamma)-x^t\| \leq (L_\hS+\eta\gamma^{-1})\|x^{t+1}(\gamma)-x^t\|. \label{L-bnd}
\end{align}
In addition, notice from \eqref{para} that $\gamma_i\leq\gamma_{i-1}\delta^{n_i-1}$ for $i=1,\ldots, t-1$, which implies that $\gamma_{t-1}\leq\gamma_0\delta^{\sum_{i=1}^{t-1} (n_i-1)}$.  Let $\gamma=\min\{\gamma_0,\delta^{-1}\gamma_{t-1}\}\delta^{M+t-\sum_{i=1}^{t-1} n_i}$.  In view of these, $\delta\in(0,1)$,  and the definition of $M$ in \eqref{nlocal1}, one can verify that 
\[
0<\gamma \leq\delta^{-1}\gamma_{t-1}\delta^{M+t-\sum_{i=1}^{t-1} n_i}\leq \gamma_0\delta^M\leq \min\{\xi/L_\hS, \gamma_0\}.
\] 
It then follows from this, \eqref{xi}, and \eqref{L-bnd} that 
\[
V_{t+1}(\gamma)\leq (L_\hS\gamma+\eta)\gamma^{-1}\|x^{t+1}(\gamma)-x^t\|\leq \nu(1-\eta)\gamma^{-1}\|x^{t+1}(\gamma)-x^t\|,
\] 
which,  together with \eqref{term}, \eqref{L-gamma}, the expression of $\gamma$, and  the definition of $n_t$ (see step 2 of Algorithm \ref{ALG}),  implies that $n_t \le M+t-\sum_{i=1}^{t-1} n_i$ and hence $x^{t+1}$ is successfully generated.
\end{proof}
 
We are now ready to prove the main results presented in Section \ref{PDE1}, namely, Theorems \ref{np1} and \ref{thm-outer}.

\begin{proof}[\textbf{Proof of Theorem \ref{np1}}]
We prove this theorem by induction. Indeed, notice from \aref{ALG} that $x^0=x^1\in \cS$. It then follows from Lemma \ref{nlemma4} that $x^2$ is successfully generated and $n_1 \le M+1$. Hence, \aref{ALG} is well-defined at iteration 1. By this, \eqref{def-tg},  and \eqref{l3b} with $k=1$, one has
\[
\|x^2-x^*\|^2\overset{\eqref{l3b}}{\leq}\frac{1}{(1-2\nu^2)\theta_1}\|x^0-x^*\|^2\overset{\eqref{def-tg}}{=}\frac{1}{1-2\nu^2}\|x^0-x^*\|^2,
\]
which together with \eqref{set-S} implies that $x^2\in\cS$. 
Now, suppose for induction that \aref{ALG} is well-defined at iteration 1 to $t-1$ and $x^i\in \cS$ for all $0 \le i \le t$ for some $2 \le t \in \bbT-1$.  It then follows from Lemma \ref{nlemma4} that $x^{t+1}$ is successfully generated and  $\sum_{i=1}^{t} n_i \leq  M+t$. Hence, \aref{ALG} is well-defined at iteration $t$. By this, \eqref{def-tg},  and \eqref{l3b} with $k=t$, one has
\[
\|x^{t+1}-x^*\|^2\overset{\eqref{l3b}}{\leq}\frac{1}{(1-2\nu^2)\theta_t}\|x^0-x^*\|^2\overset{\eqref{def-tg}}{\le}\frac{1}{1-2\nu^2}\|x^0-x^*\|^2,
\]
which together with \eqref{set-S} implies that $x^{t+1}\in\cS$. Hence, the induction is completed and the conclusion of this theorem holds.
\end{proof}

\begin{proof}[\textbf{Proof of \tref{thm-outer}}]
Notice from \aref{ALG} that $0\leq\eta<1/3$ and $0<\gamma_t \leq\gamma_0$ for all $0\leq t\in\bbT-1$. Using these and \eqref{para}, we have that for all $1 \le t \in \bbT-1$, 
\begin{align}
\gamma_t\beta_t=\gamma_{t-1}\left(1+\frac{2\mu\gamma_{t-1}}{1-\eta}\right)^{-1}\leq\gamma_0,\quad\alpha_t=\frac{\eta\gamma_t\beta_t}{\gamma_{t-1}}=\eta\left(1+\frac{2\mu\gamma_{t-1}}{1-\eta}\right)^{-1}\leq\frac{1}{3}. \label{alpha-bnd}
\end{align}

We next show by induction that
\beq \label{gammat-bnd-1}
\gamma_t  \geq \min\Big\{L_\hS^{-1}\delta\xi,\gamma_0\Big\} \quad \forall 0 \le t \in \bbT-1, 
\eeq
where $\xi$ is defined in \eqref{xi}. Indeed, \eqref{gammat-bnd-1} clearly holds at $t=0$.  Suppose that \eqref{gammat-bnd-1} holds at some $0 \le t-1 \in \bbT-2$. We now show that \eqref{gammat-bnd-1} holds at $t$ by considering the following two separate cases.

Case (a): $n_t=0$. It follows from this and \aref{ALG} that $\gamma_t=\min\{\gamma_0,\delta^{-1}\gamma_{t-1}\}$, which together with $\delta\in(0,1)$ and  \eqref{gammat-bnd-1} with $t$ replaced by $t-1$ implies that \eqref{gammat-bnd-1} holds at $t$.

Case (b): $n_t>0$. By this,  \eqref{para}, $\delta\in(0,1)$ and the definition of $n_t$, one can observe that 
$\gamma_t/\delta=\min\{\gamma_0,\delta^{-1}\gamma_{t-1}\}\delta^{n_t-1} \leq \gamma_0$ and 
 \eqref{term} will not hold if $\gamma_t$ is replaced by $\gamma_t/\delta$. Besides, from the proof of Lemma \ref{nlemma4}, one can see that  \eqref{term} will hold if $\gamma_t$ is replaced by $\tilde \gamma$  satisfying $0<\tilde \gamma \leq \min\{\xi/L_\hS, \gamma_0\}$. Hence, it follows that $\gamma_t/\delta>\xi/L_\hS$, which together with $\gamma_t \leq \gamma_0$ implies that \eqref{gammat-bnd-1} holds at $t$.

By \mref{set-S}, \mref{def-tg}, \mref{l3b}, and \eqref{gammat-bnd-1}, one has that for all $1 \le t \in \bbT-1$, 
\begin{align*}
\|x^{t+1}-x^*\|^2\leq&\ \frac{1}{(1-2\nu^2)\prod_{i=1}^{t-1}\left(1+\frac{2\mu\gamma_{i}}{1-\eta}\right)}\|x^0-x^*\|^2\notag\\
\leq&\ \frac{r^2_0}{1-2\nu^2}\left(1+\frac{2\mu}{1-\eta}\min\Big\{L_\hS^{-1}\delta\xi, \gamma_0\Big\}\right)^{1-t}.
\end{align*}
It then follows that for all $3 \le t \in \bbT-1$, 
\begin{align}
\max\{\|x^t-x^{t-1}\|,\|x^{t+1}-x^t\|\}&\leq \max\{\|x^t-x^*\|+\|x^{t-1}-x^*\|,\|x^{t+1}-x^*\|+\|x^t-x^*\|\}\notag  \\ 
& \leq \frac{2r_0}{\sqrt{1-2\nu^2}}\left(1+\frac{2\mu}{1-\eta}\min\Big\{L_\hS^{-1}\delta\xi, \gamma_0\Big\}\right)^{\frac{3-t}{2}}.\label{final}
\end{align}

Suppose for contradiction that \aref{ALG} runs for at least $\cT+1$ iterations. It then follows that 
\eqref{alg-term} fails for $t=\cT$, which along with \eqref{def-v} implies that $\|v^\cT\|>\epsilon$.
In addition, recall from Theorem \ref{np1}(ii) that $x^t \in \cS$ for all $t\in\bbT$. By this, \eqref{def-D} and Lemma \ref{F-Lipschtiz}(i), one has 
\beq 
\|\Delta^t\| =  \|F(x^t)-F(x^{t-1})\| \le L_\cS \|x^t-x^{t-1}\| \quad \forall 1 \le t \in \bbT. \label{Delta-bnd}
\eeq
Also, notice from \eqref{t1k} that $ \cT \ge 3$. By this, $\gamma_\cT \le \gamma_0$,  \eqref{t1k},  \eqref{def-v}, \eqref{alpha-bnd}, \eqref{gammat-bnd-1}, \eqref{final}, and \eqref{Delta-bnd}, one has
\begin{align*}
\|v^\cT\| \overset{\eqref{def-v}}{\le} &\ \frac{1}{\gamma_\cT}\left(\|x^{\cT+1}-x^\cT\|+\alpha_\cT \|x^\cT-x^{\cT-1}\|+\gamma_\cT\|\Delta^{\cT+1}\|+\gamma_\cT\beta_\cT\|\Delta^\cT\|\right) \\
\overset{\eqref{Delta-bnd}}{\leq}&\ \frac{1}{\gamma_\cT}\left(\|x^{\cT+1}-x^\cT\|+\alpha_\cT\|x^\cT-x^{\cT-1}\|+\gamma_\cT   L_\cS\|x^{\cT+1}-x^\cT\|+\gamma_\cT\beta_\cT L_\cS\|x^\cT-x^{\cT-1}\|\right)\\
\overset{\eqref{final}}{\leq}&\ \frac{2r_0}{\gamma_\cT\sqrt{1-2\nu^2}}\left(1+\alpha_\cT+\gamma_\cT L_\cS+\gamma_\cT\beta_\cT L_\cS\right)\left(1+\frac{2\mu}{1-\eta}\min\Big\{L_\hS^{-1}\delta\xi, \gamma_0\Big\}\right)^{\frac{3-T}{2}}\\
\overset{\eqref{alpha-bnd}}{\leq}&\ \frac{2r_0}{\gamma_\cT\sqrt{1-2\nu^2}}\left(1+\frac{1}{3}+\gamma_0L_\cS+\gamma_0L_\cS\right)\left(1+\frac{2\mu}{1-\eta}\min\Big\{L_\hS^{-1}\delta\xi, \gamma_0\Big\}\right)^{\frac{3-T}{2}}\\
\overset{\eqref{gammat-bnd-1}}{\leq}&\ \frac{r_0\left(8+12\gamma_0L_\cS\right)}{3\sqrt{1-2\nu^2}\min\left\{L_\hS^{-1}\delta\xi,\gamma_0\right\}}\left(1+\frac{2\mu}{1-\eta}\min\Big\{L_\hS^{-1}\delta\xi, \gamma_0\Big\}\right)^{\frac{3-T}{2}}\overset{\eqref{t1k}}{\leq}\epsilon,
\end{align*}
which leads to a contradiction. Hence, \aref{ALG} terminates in at most $\cT$ iterations. Suppose that \aref{ALG} terminates at iteration $t$ and outputs $x^{t+1}$ for some $t\le \cT$. It then follows that \mref{alg-term} holds for such $t$. By this and \mref{def-v}, one can see that $\|v^t\| \le \epsilon$, which together with \eqref{inclusion} implies that $\res_{F+B}(x^{t+1})\leq\epsilon$. 

Observe that $|\bbT| \le T+2$ and the total number of inner iterations of \aref{ALG} is $\sum^{|\bbT|-2}_{t=1} (n_t+1)$. It follows from these and Theorem \ref{np1}(ii) that
\[
\sum^{|\bbT|-2}_{t=1} (n_t+1) \le 2(|\bbT|-2)+M \le 2T+M, 
\]
which together with \eqref{nlocal1} and  \eqref{t1k} implies that the conclusion holds.
\end{proof}

\subsection{Proof of the main result in Section \ref{PDE2}}
\label{sec:pf-PDE2}

In this subsection we first establish several technical lemmas and then use them to prove Theorem \ref{thm2}.

Recall from Section \ref{PDE2} that $\{z^k\}_{k\in \bbK}$ denotes all the iterates generated by \aref{PPA}, where $\bbK$ is a subset of consecutive nonnegative integers starting from $0$.
Notice that at iteration $0\leq k \in\bbK-1$ of \aref{PPA}, \aref{ALG} is called to find an approximate solution of the following strongly MI problem
\begin{align}
\label{subproblem}
0\in(F_k+B)(x) = (F+B)(x)+\rho_k^{-1}(x-z^k).
\end{align}
Since $F+B$ is maximal monotone, it follows that the domain of the resolvent of $F+B$ is $\rr^n$. As a result, there exists some $z^k_*\in\rr^n$ such that
\beq
z^k_*= \left(I+\rho_k(F+B)\right)^{-1}(z^k). \label{zks}
\eeq 
Moreover, $z^k_*$ is the unique solution of problem \eqref{subproblem} and thus
\begin{align}
0 \in(F_k+B)(z^k_*) = (F+B)(z^k_*)+\rho_k^{-1}(z^k_*-z^k). \label{xstar}
\end{align}

\begin{lemma}
\label{ppal1}
Let $\{z^k\}_{k\in\bbK}$ be generated by \aref{PPA}. Then for all $0 \le k \in \bbK-1$, we have
\begin{align}
\|z^{k+1}-z_*^k\|\leq\rho_k\tau_k, \label{c3}
\end{align}
where $z_*^k$ is defined in \eqref{zks}.
\end{lemma}

\begin{proof}
By the definition of $z^{k+1}$ (see step 2 of \aref{PPA}) and Theorem \ref{thm-outer}, there exists some $v \in(F_k+B)(z^{k+1})$ with $\|v\|\leq\tau_k$. It follows from this and \eqref{xstar} that 
 \begin{align*}
 v-\rho_k^{-1}(z^{k+1}-z^k)\in(F+B)(z^{k+1}),\;\; -\rho_k^{-1}(z^k_*-z^k)\in(F+B)(z^k_*).
 \end{align*}
 By the monotonicity of $F+B$, one has
\begin{align*}
\langle v-\rho_k^{-1}(z^{k+1}-z^k)+\rho_k^{-1}(z^k_*-z^k),z^{k+1}-z^k_*\rangle\geq0,
\end{align*}
which yields 
\[
\|z^{k+1}-z^k_*\|^2\leq\rho_k\langle v,z^{k+1}-z^k_*\rangle\leq\rho_k\|v\| \|z^{k+1}-z^k_*\|.
\]
It then follows from this and $\|v\|\leq\tau_k$ that $\|z^{k+1}-z^k_*\|\leq\rho_k\|v\|\leq\rho_k\tau_k$.
\end{proof}

\begin{lemma}
\label{ppal2}
Let $\{z^k\}_{k\in\bbK}$ be generated by \aref{PPA}. Then we have
\begin{align}
\|z^s-x^*\|\leq& \ \|z^0-x^*\|+\sum_{k=0}^{s-1}\rho_k\tau_k  \qquad \forall 1 \le s \in \bbK,\label{ppae2} \\
\|z^{s+1}-z^s\|\leq&\ \|z^0-x^*\|+\sum_{k=0}^{s}\rho_k\tau_k  \qquad \forall 1 \le s \in \bbK-1 \label{zdiff-bnd}.
\end{align}
\end{lemma}

\begin{proof}
By \mref{xstar} and the definition of $x^*$, one has
\begin{align*}
z^k-z^k_* \in \rho_k (F+B)(z^k_*),\quad 0 \in \rho_k (F+B)(x^*),
\end{align*}
which together with the monotonicity of $F+B$ yield
\begin{align*}
0 \le 2 \langle z^k-z^k_*,z^k_*-x^*\rangle = \|z^k-x^*\|^2-\|z^k-z^k_*\|^2-\|z^k_*-x^*\|^2.
\end{align*}
It follows that 
\begin{align*}
\|z^k-z^k_*\|^2+\|z^k_*-x^*\|^2\leq\|z^k-x^*\|^2, 
\end{align*}
which implies that
\begin{align}
\|z^k_*-x^*\|\leq\|z^k-x^*\|,\quad\|z^k-z^k_*\|\leq\|z^k-x^*\|.\label{ineq}
\end{align}
By the first relation in \eqref{ineq}, one has
\[
\|z^{k+1}-x^*\|\leq\|z^{k+1}-z^k_*\|+\|z^k_*-x^*\| \leq\|z^{k+1}-z^k_*\|+\|z^k-x^*\|. 
\]
Summing up the above inequalities for $k=0,\dots,s-1$ yields
\[
\|z^s-x^*\|\leq\|z^0-x^*\|+\sum_{k=0}^{s-1} \|z^{k+1}-z^k_*\|, 
\]
which along with \mref{c3} implies that \eqref{ppae2} holds. In addition, using \eqref{c3} with $k=s$, \eqref{ppae2} and \eqref{ineq}, we have
\begin{align*}
\|z^{s+1}-z^s\|\leq \|z^s-z^s_*\|+\|z^{s+1}-z_*^s\|
\overset{\eqref{c3}, \eqref{ineq}}{\leq} \|z^s-x^*\|+\rho_s\tau_s \overset{\eqref{ppae2}}{\leq} \|z^0-x^*\|+\sum_{k=0}^{s}\rho_k\tau_k.
\end{align*}
Hence, \eqref{zdiff-bnd} holds as desired.
\end{proof}

Define
\begin{align}
\cS_k=& \ \left\{x\in\dom{B}:\ \|x-z^k_*\| \leq\frac{1}{\sqrt{1-2\nu^2}}\|z^k-z^k_*\| \right\} \qquad \forall 0 \le k \in \bbK-1,  \label{set-Sk} \\
\hS_k=&  \ \left\{x\in\dom{B}:\ \|x-z^k_*\|\leq 
\frac{5+9\gamma_0 L_{\cQ}}{3\sqrt{1-2\nu^2}}\|z^k-z^k_*\|\right\} \qquad \forall 0 \le k \in \bbK-1,  \label{set-hSk}
\end{align}
where $z^k_*$ is defined in \eqref{zks}, $L_\cQ$ is given in Lemma \ref{Fk-Lipschtiz}, and $\nu$ and $\gamma_0$ are the input parameters of \aref{PPA}.

\begin{lemma}\label{belong}
Let $\cS_k$ and $\hS_k$ be defined in \eqref{set-Sk} and \eqref{set-hSk}. Then for all  $0 \le k \in \bbK-1$, we have
\begin{align*}
\cS_k\subseteq\cQ,\quad\hS_k\subseteq\hQ,
\end{align*}
where $\cQ$ and $\hQ$ are defined in \mref{set-Q} and \mref{set-hQ}.  Consequently, for all $0 \le k\in \bbK-1$, $F_k$ is $L_\cQ$- and $L_\hQ$-Lipschitz continuous on $\cS_k$ and $\hS_k$, respectively, where $L_\cQ$ and $L_\hQ$ are given in Lemma \ref{Fk-Lipschtiz}.
\end{lemma}

\begin{proof}
Fix any $0 \le k\in \bbK-1$ and $x\in\cS_k$. By this, \mref{ppae2}, \mref{ineq},  \mref{set-Sk}, and $\sum_{i=0}^{\infty}\rho_i\tau_i=\rho_0\tau_0/(1-\sigma\zeta)$, we have
\begin{align*}
\|x-x^*\|\leq & \ \|x-z^k_*\|+\|z^k_*-x^*\|\overset{\mref{set-Sk}}{\leq}\frac{1}{\sqrt{1-2\nu^2}}\|z^k-z^k_*\|+\|z^k_*-x^*\| \overset{\mref{ineq}}{\leq} \left(\frac{1}{\sqrt{1-2\nu^2}}+1\right)\|z^k-x^*\|\\
\overset{\mref{ppae2}}{\leq}& \ \left(\frac{1}{\sqrt{1-2\nu^2}}+1\right)\left(\|z^0-x^*\|+\sum_{i=0}^{k-1}\rho_i\tau_i\right) 
\leq \ \left(\frac{1}{\sqrt{1-2\nu^2}}+1\right)\left(\|z^0-x^*\|+\frac{\rho_0\tau_0}{1-\sigma\zeta}\right),
\end{align*}
which together with \mref{set-Q} implies that $x\in \cQ$. It then follows that $\cS_k\subseteq\cQ$. Similarly, one can show that $\hS_k\subseteq\hQ$.
By these and the definition of $L_\cQ$ and $L_\hQ$ in Lemma \ref{Fk-Lipschtiz}, one can see that $F_k$ is $L_\cQ$- and $L_\hQ$-Lipschitz continuous on $\cS_k$ and $\hS_k$, respectively.
\end{proof}

\begin{lemma} \label{sub-cplx}
Let $\gamma_0$, $\delta$, $\nu$, $\eta$, $\{\rho_k\}$ and $\{\tau_k\}$ be given in \aref{PPA}, and let $\xi$, $\br_0$ and $\Lambda$ be defined in \eqref{xi}, \eqref{set-Q} and \eqref{C1}. 
Then for any $0 \le k \in \bbK-1$, the number of evaluations of $F$ and resolvent of $B$ performed in the $k$th iteration of \aref{PPA} is at most $\mm_k$, where
\begin{align}
\mm_k=6+2\left\lceil\frac{2\log{\frac{(\br_0+\Lambda)\left(8+12\gamma_0L_\cQ\right)}{3\tau_k\sqrt{1-2\nu^2}\min\left\{L_\hQ^{-1}\delta\xi,\gamma_0\right\}}}}{\log\left(1+\frac{2}{\rho_k(1-\eta)}\min\Big\{L_\hQ^{-1}\delta\xi, \gamma_0\Big\}\right)}\right\rceil_++\left\lceil\frac{\log\left(\frac{\xi}{\gamma_0 L_\hQ}\right)}{\log\delta}\right\rceil_+. \label{t2M}
\end{align}

\end{lemma}

\begin{proof}
Recall that $F_k+B$ is $1/\rho_k$-strongly monotone and $(F_k+B)^{-1}(0)\neq\emptyset$.
In addition, it follows from \lref{belong} that $F_k$ is $L_\cQ$- and $L_\hQ$-Lipschitz continuous on $\cS_k$ and $\hS_k$, respectively.  Using \mref{ppae2}, \mref{ineq},  and the fact that $\Lambda=\sum_{k=0}^{\infty}\rho_k\tau_k$, we have
\begin{align*}
\|z^k-z^k_*\|\leq\|z^k-x^*\|\leq\|z^0-x^*\|+\sum_{t=0}^{k-1}\rho_t\tau_t\leq \br_0+\Lambda,
\end{align*}
where $z^k_*$ is given in \eqref{zks}. The conclusion then follows from applying Theorem \ref{thm-outer} to the subproblem \mref{subproblem} with $\epsilon$, $\mu$, $L_\cS$, $L_\hS$, and $r_0$ in \eqref{t1n} being replaced by $\tau_k$, $1/\rho_k$, $L_\cQ$, $L_\hQ$, and $\br_0+\Lambda$, respectively. 
\end{proof}

\begin{proof}[\textbf{Proof of \tref{thm2}}]
Suppose for contradiction that \aref{PPA} runs for more than $K+1$ outer iterations. By this and   \aref{PPA}, one can assert that \eqref{ppa-term} does not hold for $k=K$. On the other hand, 
by \eqref{zdiff-bnd}, $\rho_k=\rho_0\zeta^k$ and $\tau_k=\tau_0\sigma^k$, one has
\begin{align*}
\frac{\|z^{K+1}-z^K\|}{\rho_K} \leq \frac{\|z^0-x^*\|+\sum_{k=0}^{K}\rho_k\tau_k}{\rho_K}\leq \frac{\|z^0-x^*\|+\sum_{k=0}^{\infty}\rho_k\tau_k}{\rho_K} = \frac{\br_0+\Lambda}{\rho_0 \zeta^K}.
\end{align*}
This together with the definition of $K$ in \mref{t2N} implies that
\begin{align*}
\frac{\|z^{K+1}-z^K\|}{\rho_K}\leq\frac{\br_0+\Lambda}{\rho_0\zeta^{K}}\leq \frac{\varepsilon}{2},\quad\tau_K=\tau_0\sigma^K\leq\frac{\varepsilon}{2},
\end{align*}
and thus \eqref{ppa-term} holds for $k=K$, which contradicts the above assertion. Hence, \aref{PPA} must terminate in at most $K+1$ outer iterations. 

Suppose that \aref{PPA} terminates at some iteration $k$. Then we have
\begin{align}
\rho_k^{-1}\|z^{k+1}-z^k\| +  \tau_k\leq \varepsilon. \label{bdd1}
\end{align}
In addition, by the definition of $z^{k+1}$ (see step 2 of \aref{PPA}), Theorem \ref{thm-outer}, and \eqref{Fk}, there exists some $v$ such that
\begin{align}
v\in(F+B)(z^{k+1})+\rho_k^{-1}(z^{k+1}-z^k),\quad \|v\|\leq\tau_k. \label{bdd2}
\end{align}
Observe that $v-\rho_k^{-1}(z^{k+1}-z^k)\in(F+B)(z^{k+1})$. It follows from this, 
\eqref{bdd1}, \eqref{bdd2}, and the definition of $\res_{F+B}$ that 
\begin{align*}
\res_{F+B}(z^{k+1})\le\|v-\rho_k^{-1}(z^{k+1}-z^k)\|\leq\|v\|+\rho_k^{-1}\|z^{k+1}-z^k\|\overset{\eqref{bdd2}}{\le}\tau_k+\rho_k^{-1}\|z^{k+1}-z^k\| \overset{\eqref{bdd1}}{\le} \varepsilon.
\end{align*}

Recall from Lemma \ref{sub-cplx} that the number of evaluations of $F$ and resolvent of $B$ performed in the $k$th iteration of \aref{PPA} is at most $\mm_k$, where $\mm_k$ is given in \eqref{t2M}. In addition, using \eqref{C1}, \eqref{C2} and \mref{t2M}, we have
\begin{align}
\mm_k=&6+2\left\lceil\frac{2C_1 -2k \log\sigma}{\log\left(\min\left\{1+\frac{2\delta\xi}{L_\hQ\rho_0(1-\eta)\zeta^k}, 1+\frac{2\gamma_0}{\rho_0(1-\eta)\zeta^k}\right\}\right)}\right\rceil_++C_2. \label{Mk}
\end{align}
By the concavity of $\log(1+y)$, one has $\log(1+\vartheta y) \ge \vartheta \log (1+y)$ for all $y>-1$ and 
$\vartheta \in[0,1]$. Using this, we obtain that 
\begin{align}
&\log\left(\min\left\{1+\frac{2\delta\xi}{\rho_0(1-\eta)\zeta^kL_\hQ}, 1+\frac{2\gamma_0}{\rho_0(1-\eta)\zeta^k}\right\}\right) \notag \\ 
&=\ \min\left\{ \log\left(1+\zeta^{-k}\frac{2\delta\xi}{\rho_0(1-\eta)L_\hQ}\right), \log\left(1+\zeta^{-k}\frac{2\gamma_0}{\rho_0(1-\eta)}\right)\right\} \notag \\
& \ge \ \min\left\{\zeta^{-k} \log\left(1+\frac{2\delta\xi}{\rho_0(1-\eta)L_\hQ}\right), \zeta^{-k}\log\left(1+\frac{2\gamma_0}{\rho_0(1-\eta)}\right)\right\}.  \label{log-bnd}
\end{align}

By \eqref{Mk}, \eqref{log-bnd}, $\sigma \in (0,1)$ and $\zeta>1$, one has that for all $k\geq0$,
\begin{align}
\mm_k \overset{\eqref{Mk}}{\le} &\ 8+\frac{\left(4C_1 -4k \log\sigma\right)_+}{\log\left(\min\left\{1+\frac{2\delta\xi}{\rho_0(1-\eta)\zeta^kL_\hQ}, 1+\frac{2\gamma_0}{\rho_0(1-\eta)\zeta^k}\right\}\right)}+C_2 \notag \\ 
\le &  8+\frac{4\left(C_1\right)_+-4k\log\sigma}
{\log\left(\min\left\{1+\frac{2\delta\xi}{\rho_0(1-\eta)\zeta^kL_\hQ}, 1+\frac{2\gamma_0}{\rho_0(1-\eta)\zeta^k}\right\}\right)}+C_2\notag \\ 
\overset{\eqref{log-bnd}}{\le} & \ 8+\frac{4\zeta^k\left( C_1\right)_+-4k\zeta^k\log\sigma}
{\log\left(\min\left\{1+\frac{2\delta\xi}{\rho_0(1-\eta)L_\hQ}, 1+\frac{2\gamma_0}{\rho_0(1-\eta)}\right\}\right)}+C_2.\label{Mk-bound}
\end{align}

Observe that $|\bbK|\leq K+2$ and also the total number of inner iterations of \aref{PPA} is at most $\sum^{|\bbK|-2}_{t=0} \mm_t$. It then follows from \eqref{C2}, \eqref{Mk} and \eqref{Mk-bound} that the total number of evaluations of $F$ and resolvent of $B$ performed in \aref{PPA} is at most
\begin{align}
\sum_{k=0}^{|\bbK|-2}\mm_k \le& \ \sum_{k=0}^{K}\left(8+\frac{4\zeta^k\left(C_1\right)_+-4k\zeta^k\log\sigma}
{\log\left(\min\left\{1+\frac{2\delta\xi}{\rho_0(1-\eta)L_\hQ}, 1+\frac{2\gamma_0}{\rho_0(1-\eta)}\right\}\right)}+C_2\right) \notag \\
\le&\ 8K+8+4\left(C_1\right)_+C_3\zeta^{K+1}+4C_3(\log\sigma^{-1}) K\zeta^{K+1}+(K+1)C_2, \label{sum-Mk}
\end{align}
where the second inequality is due to $\sum_{k=0}^K\zeta^k\leq\zeta^{K+1}/(\zeta-1)$ and $\sum_{k=0}^Kk\zeta^k\leq K\zeta^{K+1}/(\zeta-1)$. By the definition of $K$ in \eqref{t2N}, one has
\[
K \le  \max \left\{\log_{\zeta}\left(\frac{2\br_0+2\Lambda}{\varepsilon\rho_0}\right)+1,\frac{\log(2\tau_0/\varepsilon)}{\log(1/\sigma)}+1,0\right\}, 
\]
which together with $\zeta>1$ implies that 
\begin{align}
\zeta^K \le  & \ \max\left\{\frac{2\zeta(\br_0+\Lambda)}{\varepsilon\rho_0},\zeta\left(\frac{2\tau_0}{\varepsilon}\right)^{\frac{\log\zeta}{\log(1/\sigma)}},1\right\}. \label{zeta-bnd}
\end{align}
Using  \eqref{t2N}, \eqref{complexity}, \eqref{sum-Mk} and \eqref{zeta-bnd}, we can see that $\sum_{k=0}^{|\bbK|-2}\mm_k \le \mm$.
\end{proof}

\section{Concluding remarks}\label{sec:cr}
We proposed primal-dual extrapolation methods enjoying an operation complexity of $\oo{\log\varepsilon^{-1}}$ and $\oo{\varepsilon^{-1}\log \varepsilon^{-1}}$, measured by the number of fundamental operations for finding an $\epsilon$-residual solution of strongly and non-strongly monotone inclusion problems under local Lipschitz continuity, respectively. The latter complexity significantly improves upon the previously best operation complexity $\cO(\varepsilon^{-2})$ achieved by the FRBS method \cite{malitsky2020forward}. 

One natural question is whether the aforementioned operation complexity of $\oo{\varepsilon^{-1}\log \varepsilon^{-1}}$ can be improved to $\oo{\varepsilon^{-1}}$, which would match the optimal complexity for solving non-strongly monotone inclusion problems under global Lipschitz continuity. Additionally, our proposed methods require the exact resolvent of $B$, which limits their applicability.  Clearly,  a method using the inexact resolvent of $B$ for the monotone inclusion problem would be both practically and theoretically interesting. It is worthwhile to explore these as future research directions.


\begin{thebibliography}{10}

\bibitem{boct2016inertial}
R.~I. Bo{\c{t}} and E.~R. Csetnek.
\newblock An inertial forward-backward-forward primal-dual splitting algorithm
  for solving monotone inclusion problems.
\newblock {\em Numerical Algorithms}, 71:519--540, 2016.

\bibitem{FaPa07}
F.~Facchinei and J.~S. Pang.
\newblock {\em Finite-Dimensional Variational Inequalities and Complementarity
  Problems}.
\newblock Springer Science \& Business Media, 2007.

\bibitem{huang2021unifying}
K.~Huang and S.~Zhang.
\newblock A unifying framework of accelerated first-order approach to strongly
  monotone variational inequalities.
\newblock {\em arXiv preprint arXiv:2103.15270}, 2021.

\bibitem{huang2022new}
K.~Huang and S.~Zhang.
\newblock New first-order algorithms for stochastic variational inequalities.
\newblock {\em SIAM Journal on Optimization}, 32(4):2745--2772, 2022.

\bibitem{Kor76}
G.~M. Korpelevich.
\newblock Extragradient method for finding saddle points and other problems.
\newblock {\em Ekonomika i Matem.\ Metody}, 12:747--756, 1976.

\bibitem{kotsalis2020simple}
G.~Kotsalis, G.~Lan, and T.~Li.
\newblock Simple and optimal methods for stochastic variational inequalities,
  {I}: operator extrapolation.
\newblock {\em SIAM Journal on Optimization}, 32(3):2041--2073, 2022.

\bibitem{kovalev2022first}
D.~Kovalev and A.~Gasnikov.
\newblock The first optimal algorithm for smooth and
  strongly-convex-strongly-concave minimax optimization.
\newblock {\em Advances in Neural Information Processing Systems},
  35:14691--14703, 2022.

\bibitem{latafat2023adaptive}
P.~Latafat, A.~Themelis, L.~Stella, and P.~Patrinos.
\newblock Adaptive proximal algorithms for convex optimization under local
  lipschitz continuity of the gradient.
\newblock {\em arXiv preprint arXiv:2301.04431}, page~4, 2023.

\bibitem{Lin20}
T.~Lin, C.~Jin, and M.~I. Jordan.
\newblock Near-optimal algorithms for minimax optimization.
\newblock In {\em Proceedings of Machine Learning Research}, pages 1--42, 2020.

\bibitem{lions1979splitting}
P.-L. Lions and B.~Mercier.
\newblock Splitting algorithms for the sum of two nonlinear operators.
\newblock {\em SIAM Journal on Numerical Analysis}, 16(6):964--979, 1979.

\bibitem{lu2023iteration}
Z.~Lu and Z.~Zhou.
\newblock Iteration-complexity of first-order augmented {L}agrangian methods
  for convex conic programming.
\newblock {\em SIAM Journal on Optimization}, 33(2):1159--1190, 2023.

\bibitem{malitsky2020golden}
Y.~Malitsky.
\newblock Golden ratio algorithms for variational inequalities.
\newblock {\em Mathematical Programming}, 184(1-2):383--410, 2020.

\bibitem{malitsky2020forward}
Y.~Malitsky and M.~K. Tam.
\newblock A forward-backward splitting method for monotone inclusions without
  cocoercivity.
\newblock {\em SIAM Journal on Optimization}, 30(2):1451--1472, 2020.

\bibitem{Mok20-2}
A.~Mokhtari, A.~E. Ozdaglar, and S.~Pattathil.
\newblock Convergence rate of ${O}(1/k)$ for optimistic gradient and
  extragradient methods in smooth convex-concave saddle point problems.
\newblock {\em SIAM Journal on Optimization}, 30:3230--3251, 2020.

\bibitem{monteiro2010complexity}
R.~D. Monteiro and B.~F. Svaiter.
\newblock On the complexity of the hybrid proximal extragradient method for the
  iterates and the ergodic mean.
\newblock {\em SIAM Journal on Optimization}, 20(6):2755--2787, 2010.

\bibitem{monteiro2011complexity}
R.~D. Monteiro and B.~F. Svaiter.
\newblock {Complexity of variants of Tseng's modified FB splitting and
  Korpelevich's methods for hemivariational inequalities with applications to
  saddle-point and convex optimization problems}.
\newblock {\em SIAM Journal on Optimization}, 21(4):1688--1720, 2011.

\bibitem{moudafi2003convergence}
A.~Moudafi and M.~Oliny.
\newblock Convergence of a splitting inertial proximal method for monotone
  operators.
\newblock {\em Journal of Computational and Applied Mathematics},
  155(2):447--454, 2003.

\bibitem{Nem05}
A.~Nemirovski.
\newblock Prox-method with rate of convergence ${O}(1/t)$ for variational
  inequalities with {Lipschitz} continuous monotone operators and smooth
  convex-concave saddle-point problems.
\newblock {\em SIAM Journal on Optimization}, pages 229--251, 2005.

\bibitem{Nest03-3}
Y.~E. Nesterov.
\newblock Dual extrapolation and its applications to solving variational
  inequalities and related problems.
\newblock {\em Mathematical Programming}, 109:319--344, 2003.

\bibitem{passty1979ergodic}
G.~B. Passty.
\newblock {Ergodic convergence to a zero of the sum of monotone operators in
  Hilbert space}.
\newblock {\em Journal of Mathematical Analysis and Applications},
  72(2):383--390, 1979.

\bibitem{popov1980modification}
L.~D. Popov.
\newblock A modification of the {A}rrow-{H}urwicz method for search of saddle
  points.
\newblock {\em Mathematical notes of the Academy of Sciences of the USSR},
  28(5):845--848, 1980.

\bibitem{Rocka76}
R.~T. Rockafellar.
\newblock Monotone operators and the proximal point algorithm.
\newblock {\em SIAM Journal on Control and Optimization}, 14:877--898, 1976.

\bibitem{Sib70}
M.~Sibony.
\newblock M{\'e}thodes it{\'e}ratives pour les {\'e}quations et in{\'e}quations
  aux d{\'e}riv{\'e}es partielles non lin{\'e}aires de type monotone.
\newblock {\em CALCOLO}, 7:65--183, 1970.

\bibitem{tran2022connection}
Q.~Tran-Dinh.
\newblock The connection between {Nesterov's} accelerated methods and {Halpern}
  fixed-point iterations.
\newblock {\em arXiv preprint arXiv:2203.04869}, 2022.

\bibitem{tseng2000modified}
P.~Tseng.
\newblock A modified forward-backward splitting method for maximal monotone
  mappings.
\newblock {\em SIAM Journal on Control and Optimization}, 38(2):431--446, 2000.

\bibitem{Tse08}
P.~Tseng.
\newblock On accelerated proximal gradient methods for convex-concave
  optimization.
\newblock Manuscript, May 2008.

\bibitem{Tom21}
T.~Vladislav, T.~Yaroslav, B.~Ekaterina, K.~Dmitry, A.~Gasnikov, and
  P.~Dvurechensky.
\newblock On accelerated methods for saddle-point problems with composite
  structure.
\newblock {\em arXiv preprint arXiv:2103.09344}, 2021.

\bibitem{Yang20}
J.~Yang, S.~Zhang, N.~Kiyavash, and N.~He.
\newblock A catalyst framework for minimax optimization.
\newblock In {\em Advances in Neural Information Processing Systems}, pages
  5667--5678, 2020.

\bibitem{pmlr-v139-yoon21d}
T.~Yoon and E.~K. Ryu.
\newblock Accelerated algorithms for smooth convex-concave minimax problems
  with ${O}(1/k^2)$ rate on squared gradient norm.
\newblock In {\em Proceedings of the 38th International Conference on Machine
  Learning}, volume 139, pages 12098--12109, 2021.

\bibitem{zhang2023primal}
J.~Zhang, M.~Wang, M.~Hong, and S.~Zhang.
\newblock Primal-dual first-order methods for affinely constrained multi-block
  saddle point problems.
\newblock {\em SIAM Journal on Optimization}, 33(2):1035--1060, 2023.

\end{thebibliography}
\end{document}